\documentclass[11pt,reqno,english,empty]{amsart}
\usepackage{array}
\usepackage{frcursive}
\usepackage[T1]{fontenc}
\usepackage{mathrsfs}
\usepackage{amsmath,amsthm,amssymb}
\usepackage{latexsym}
\usepackage{enumerate}
\usepackage{mathrsfs}
\usepackage{stmaryrd}
\usepackage{amsopn}
\usepackage{amsmath}
\usepackage{amssymb}
\usepackage{amsfonts}
\usepackage{soul}
\usepackage{amsbsy}
\usepackage{amscd,indentfirst,epsfig}
\usepackage{amsfonts,amsmath,latexsym,amssymb,verbatim,amsbsy}
\usepackage{amsthm}
\usepackage{colordvi}
\usepackage{pstricks}
\usepackage{subfigure}

\def\R{\mathbb R}

\def\N{\mathbb N}

\def\C{\mathcal C}

\def\P{\mathcal P}

\def \M{\mathcal{M}}

\def\a{a}
\def\b{b}
\def \V {\mathcal{V}}

\def \ra {\rangle}
\def \la {\langle}

 \def \ds {\displaystyle}

\def \nub {\boldsymbol \nu}

\def \mub {\boldsymbol \mu}
\def\d{\mathrm{d}}

\def\Q{\mathcal{Q}}

\newtheorem{theo}{Theorem}[section]
\newtheorem{prop}[theo]{Proposition}
\newtheorem{cor}[theo]{Corollary}
\newtheorem{lem}[theo]{Lemma}
\newtheorem{defi}[theo]{Definition}
\newtheorem{nb}[theo]{Remark}
\def \leq {\leqslant}
\def \geq {\geqslant}
\def \gmm {\Upsilon}%{\boldsymbol \gamma}}

\numberwithin{equation}{section}
\def\beq{\begin{equation}}
\def\eeq{\end{equation}}
\def\beqn{\begin{equation*}}
\def\eeqn{\end{equation*}}
\def\bea{\begin{eqnarray}}
\def\eea{\end{eqnarray}}
\def\bean{\begin{eqnarray*}}
\def\eean{\end{eqnarray*}}
\def\bary{\begin{array}}
\def\eary{\end{array}}

\title[Dissipative Boltzmann equation in 1-D]
{\textbf{One dimensional dissipative Boltzmann equation: measure solutions, cooling rate and self-similar profile}}

\author{R. Alonso, V. Bagland, Y. Cheng \& B. Lods }

\address{\textbf{Ricardo J. Alonso}, Departamento de Matem\'{a}tica, PUC-Rio, Rua Marqu\^{e}s de S\~ao Vicente 225, Rio de Janeiro, CEP 22451-900, Brazil.} \email{ralonso@mat.puc-rio.br}

\address{\textbf{Véronique Bagland}, Clermont Universit\'e, Universit\'{e} Blaise Pascal, Laboratoire de Math\'{e}matiques, CNRS UMR 6620,  BP 10448, F-63000 Clermont-Ferrand,
France.}\email{Veronique.Bagland@math.univ-bpclermont.fr}

\address{\textbf{Yingda Cheng}, Department of Mathematics, Michigan State University, East Lansing, MI 48824 U.S.A.}\email{ycheng@math.msu.edu}

\address{\textbf{Bertrand Lods}, Universit\`{a} degli
Studi di Torino \& Collegio Carlo Alberto, Department of Economics and Statistics, Corso Unione Sovietica, 218/bis, 10134 Torino, Italy.}\email{bertrand.lods@unito.it}

\begin{document}

\maketitle

\begin{abstract}
This manuscript investigates the following aspects of the one dimensional dissipative Boltzmann equation associated to variable hard-spheres kernel: (1) we show the optimal cooling rate of the model by a careful study of the system satisfied by the solution's moments, (2) give existence and uniqueness of measure solutions, and (3) prove the existence of a non-trivial self-similar profile, i.e. homogeneous cooling state, after appropriate scaling of the equation.  The latter issue is based on compactness tools in the set of Borel measures.  More specifically, we apply a dynamical fixed point theorem on a suitable stable set, for the model dynamics, of Borel measures.

\medskip


\end{abstract}

\tableofcontents

\section{Introduction}\label{sec:intro}
In this document we study a standard one-dimensional dissipative Boltzmann equation associated to "variable hard potentials" interaction kernels.  The model is given by
\begin{equation}\begin{split}\label{eq:1}
\partial_{t} f(t,x)&=\Q(f,f)(t,x), \qquad (t,x)\in [0,\infty)\times\mathbb{R}\,,\\
f(0,x)&=f_{0}(x),
\end{split}\end{equation}
where the dissipative Boltzmann operator $\Q=\Q_{\gamma}$ is defined as
\begin{equation}
\label{eq:cheng}
\Q(f,f)(x):=\int_{\R}f\left(x-\a y\right)f\left(x+\b y\right)\,|y|^{\gamma}\d y-f(x)\int_{\R}f(x+y)\,|y|^{\gamma}\d y\,.
\end{equation}
The parameters of the model satisfy $\gamma > 0$, $\a \in (0,1)$, $\b=1-\a$ and will be \emph{fixed throughout the paper}. Such a model can be seen as a generalisation of the one introduced by Ben-Naïm
and Krapivsky \cite{Ben-Naim1} for $\gamma=0$ and 
happens to have many applications in physics, biology and economy, see for instance the process presented in \cite{AT,Ben-Naim1} with application to biology. 

The case $\gamma=0$ -- usually referred to as the Maxwellian interaction case -- is by now well understood \cite{CarrTo} and we will focus our efforts in extending several of the results known for that case ($\gamma=0$) to the more general model \eqref{eq:1}. Let us recall that, generally speaking, the analysis of  Boltzmann-like models with Maxwellian interaction essentially renders explicit formulas that allow for a very precise analysis  \cite{BC, BCT, BCG, pareschi, CarrTo, Ben-Naim1} because (1) moments solve closed ODEs and, (2) Fourier transform techniques are  relatively simple to implement.   In \cite{Ben-Naim1} a Brownian thermalization is added to the equation which permits a study of stationary solutions of \eqref{eq:1} for $\gamma=0$.  Here we are more interested in generalizing the works of \cite{BC, BCT, BCG, pareschi,CarrTo} that deal with the so-called self-similar profile which in the particular case of Mawellian interactions is unique and explicit.  Such a self-similar profile is the unique stationary solution of the \textit{self-similar equation} associated to \eqref{eq:1} and, by means of a suitable Fourier metric, it is possible to show (in the Maxwellian case) exponential convergence of the (time dependent) self-similar solution to this stationary profile, \cite{BCT, BCG, CarrTo}.  In particular, this means that solutions of \eqref{eq:1} with $\gamma=0$ approach exponentially fast the "back rescaled" self-similar profile as $t\to+\infty$.

Self-similarity is a general feature of dissipative collision-like equations.  Indeed, since the kinetic energy is continuously decreasing, solutions to \eqref{eq:1} converge as $t \to \infty$ towards a Dirac mass. As a consequence, one expects that a suitable time-velocity scale depending on the rate of dissipation of energy may render a better set up for the analysis.  For this reason, we expect that several of the results occurring for Maxwellian interactions remain still valid for \eqref{eq:1} with $\gamma >0$.  The organization of the document is as follows:  we finish this introductory material with general set up of the problem, including notation, scaling, relevant comments and the statement of the main results.  In Section 2 the Cauchy problem is studied.  The framework will be the space of probability Borel measures.  Such a framework is the natural one for equation \eqref{eq:1} as it is for kinetic models in general.  In one dimensional problems, however, we will discover that it is essential to work in this space in contrast to higher dimensional models, such as viscoelastic Boltzmann models in the plane or the space where one can avoid it and work in smaller spaces such as Lebesgue's spaces \cite{BCaG, villani, AL1, MiMo}.  This last fact proves to be a major difficulty in the analysis of the model.  The Cauchy problem is then based on a careful study of \textit{a priori} estimates for the moments of solutions of equation \eqref{eq:1} and standard fixed point theory.  In Section 3 we find the optimal rate of dissipation (commonly referred as Haff's law) which follows from a careful study of a lower bound for the moments using a technique introduced in \cite{AL1, AL2}.  In Section 4 we prove the existence of a non-trivial self-similar profile which is based, again, on the theory of moments and the use of  a novel dynamical fixed point result on a compact stable set of Borel probability measures \cite{BagLau07}. The key remaining argument is, then, to prove that the self-similar profile - which \emph{ a priori} is a measure - is actually an $L^{1}$-function. Needless to say, such stable set is engineered out of the moment analysis of Sections 2 \& 3.   In Section 5, numerical simulations are presented that illustrate the previous quantitative study of equation \eqref{gs} as well as some peculiar features of the self-similar profile $G$.  The simulations are based upon a discontinuous Galerkin (DG) scheme. The paper ends with some perspectives and open problems related to \eqref{eq:1}, in particular, its link to a recent kinetic model for rods alignment \cite{AT, Ben-Naim}.
\subsection{Self-similar equation and the long time asymptotic}
The weak formulation of the collision operator $\Q$ reads
\begin{align}\label{eq:Qw}
\begin{split}
\int_{\R} \Q\big(f,&f\big)(x)\psi(x)\d x \\
&=\dfrac{1}{2}\int_{\R}\int_{\R}f(x)f(y)\big|x-y\big|^{\gamma}\,\Big(\psi(\a x+\b y)+\psi(\b x+\a y)-\psi(x)-\psi(y)\Big)\d x\d y
\end{split}
\end{align}
for any suitable test function $\psi$.  In particular, plugging successively $\psi(x)=1$ and $\psi(x)=x$ into \eqref{eq:Qw} shows that the above equation \eqref{eq:1} conserves mass and momentum. Namely, for any reasonable solution $f(t,x)$ to \eqref{eq:1}, one has
\begin{equation} 
\label{eq:conservation}
\int_{\R} f(t,x)\d x=\int_{\R} f_{0}(x)\d x \qquad \text{ and } \qquad \int_{\R} x f(t,x)d x=\int_{\R} x f_{0}(x)\d x \qquad \forall t \geq 0\,.
\end{equation}
However, the second order moment is not conserved: indeed, plugging now $\psi(x)=|x|^{2}$ into \eqref{eq:Qw} one sees that
\begin{equation*}
\int_{\R}\Q(f,f)(x)\,|x|^{2}\d x=- \a\,\b\int_{\R^{2}} f(x)f(y)\,|x-y|^{2+\gamma}\d x\d y
\end{equation*}
since $|\a x+\b y|^{2}+|\a y+\b x|^{2}-|x|^{2}-|y|^{2}=-2\a\,\b\,|x-y|^{2}$ for any $(x,y) \in \R^{2}$.  Therefore, for any nonnegative solution $f(t,x)$ to \eqref{eq:1}, we get that the kinetic energy is non increasing
\begin{equation}
\label{kin}
\dfrac{\d }{\d t}E(t):=\dfrac{\d}{\d t}\int_{\R} f(t,x)\,|x|^{2}\d x=-\a\,\b\int_{\R^{2}}f(t,x)\,f(t,y)\,|x-y|^{2+\gamma}\d x \d y \leq 0\,.
\end{equation}
This is enough to prove that a non trivial stationary solution to problem \eqref{eq:1} exists.  Indeed, for any $x_{0} \in \R$, the Dirac mass $\delta_{x_{0}}$ is a steady (measure) solution to \eqref{eq:1}.  For this reason one expects the large-time behavior of the system to be described by self-similar solutions.  In order to capture such a self-similar behavior, it is customary to introduce the rescaling
\begin{equation}\label{eq:change}
V(t)\,g(s(t),\xi) = f(t,x)\,,\quad \xi=V(t)\,x
\end{equation} 
where $V(t)$ and $s(t)$ are strictly increasing functions of time satisfying $V(0)=1$, $s(0)=0$ and $\lim_{t \to \infty}s(t)=\infty$.  Under such a scaling, one computes the \textit{self-similar} equation as
\begin{equation*}
\begin{split}
\partial_{t}f(t,x)&=\left(\dot{V}(t)g(s,\xi)+V(t)\dot{s}(t)\partial_{s}g(s,\xi)+{\xi}\dot{V}(t)\partial_{\xi}g(s,\xi)\right)\bigg|_{s=s(t)\,,\,\xi=V(t)x}\\
&=\left(\dot{V}(t)\partial_{\xi}\left(\xi\,g(s,\xi)\right)+V(t)\dot{s}(t)\partial_{s}g(s,\xi)\right)\bigg|_{s=s(t)\,,\,\xi=V(t)x}
\end{split}
\end{equation*}
while the interaction operator turns into
\begin{equation}
\Q\big(f,f\big)(t,x)=V^{1-\gamma}(t)\Q\big(g,g\big)(s(t),V(t)x)\,.
\end{equation}
Consequently $f=f(t,x)$ is a solution to \eqref{eq:1} if and only if $g=g(s,\xi)$ satisfies
\begin{align}\label{Rodmr}
\dot{s}(t)V^{\gamma}(t)\partial_{s}g(s,\xi)  + \frac{\dot{V}(t)}{V^{1-\gamma}(t)}\partial_{\xi}\big(\xi\,g\big)(s,\xi) = \Q\big(g,g\big)(s,\xi).\nonumber
\end{align}
Choosing 
\begin{equation*}
V(t)= \left(1+ c\gamma\,t\right)^{\frac{1}{\gamma}}\;\; \quad
\text{ and } \qquad s(t)= \frac{1}{c\gamma} \log(1+ c\,\gamma\,t)\,, \qquad c > 0,
 \end{equation*}
it follows that $g$ solves
\begin{equation}\label{gs}
\partial_{s}g(s,\xi) + c\, \partial_{\xi}\left(\xi g(s,\xi)\right)=\Q\big(g,g\big)(s,\xi)\,.
\end{equation}
Thus, the argument of understanding the long time asymptotic of \eqref{eq:1} is simple: if there exists a unique steady solution $G$ to \eqref{gs}, then such a steady state $G$ should attract any solution to \eqref{gs} and, back scaling to the original variables
\begin{equation*}
f(t,x) \simeq V(t)G(V(t)x) \text{ as } {{t \to \infty}}\,.
\end{equation*}
in some suitable topology.  We give in this paper a first step towards a satisfactory answer to this problem; more specifically, we address here two main questions:
\begin{enumerate}[\textbf{Question} 1.]
\item  Determine  the optimal convergence rate of solutions to \eqref{eq:1} towards the Dirac mass centred in the center of mass $\overline{x}_{0}:=\int_{\R} x f_{0}(x)\d x$. The determination of this optimal convergence rate is achieved by identifying the \emph{optimal rate} of convergence of the moments $m_{k}(t)$ of $f(t,x)$ defined as
\begin{equation*}
m_{k}(t):=\int_{\R}|x-\overline{x}_{0}|^{k}\,f(t,x)\d x\,\qquad k \geq 0.
\end{equation*}
\item Prove the existence of a ``physical'' steady solution $G \in L^{1}_{\max(\gamma,2)}(\R)$ to \eqref{gs}, that is a function satisfying
\begin{equation}\label{steady}
c\,\dfrac{\d}{\d {\xi}} \left(\xi\,G(\xi)\right)=\Q(G,G)(\xi)\,, \qquad \xi \in \R\,,
\end{equation}
in a weak sense (where $c >0$ is arbitrary and, for simplicity, can be chosen as $c=1$). Note that equation \eqref{steady} has at least two solutions for any $\gamma>0$, the trivial one and the Dirac measure at zero.  None of them is a relevant steady solution  since both have energy zero which is a feature not satisfied by the dynamical evolution of \eqref{gs} (provided the initial measure is neither the trivial measure nor the Dirac measure).  
\end{enumerate}
Similar questions have already been addressed for the 3-D Boltzmann equation for granular gases with different type of forcing terms \cite{MiMo,GaPaVi,BiCaLo}.  For the inelastic Boltzmann equation in $\R^{3}$, the answer to \textbf{Question 1} is known as Haff's law proven in \cite{MiMo,AL1,AL2} for the interesting case of hard-spheres interactions (essentially the case $\gamma=1$). The method we adopt here is inspired by the last two references since it appears to be the  most natural to the equation.  Concerning \textbf{Question 2}, the existence and uniqueness (the latter in a weak inelastic regime) of solutions to \eqref{steady} has been established rigorously for hard-spheres interactions in \cite{MiMo,MiMo2}.  In the references \cite{MiMo,GaPaVi,BiCaLo}, also \cite{BagLau07, MiEs} in the context of coagulation problems, the strategy to prove the existence of solutions to problem similar to \eqref{steady} is achieved through the careful study of the associated evolution equation (equation \eqref{gs} in our context) and an application of the following dynamic version of \textit{Tykhonov fixed point theorem} (see \cite[Appendix A]{BagLau07} for a proof):
\begin{theo}[\textit{\textbf{Dynamic fixed point theorem}}]\label{GPV}
Let $\mathcal{Y}$ be a locally convex topological vector space and  $\mathcal{Z}$ a nonempty convex and  compact subset of\, ${\mathcal{Y}}$. If $(\mathcal{F}_t)_{t \ge 0}$ is  a continuous semi-group on $\mathcal{Z}$ such that $\mathcal{Z}$ is invariant under the action of $\mathcal{F}_t$  (that is $\mathcal{F}_t z  \in\mathcal{Z}$ for any $z \in {\mathcal{Z}}$ and $t \ge 0$), then, there exists $z_o \in \mathcal{Z}$ which is stationary under the action of $\mathcal{F}_t$ (that is $\mathcal{F}_t z_o=z_o$ for any $t \ge 0$).
\end{theo}

In the aforementioned references, the natural approach consists in applying Theorem \ref{GPV} to $\mathcal{Y}=L^{1}$ endowed with its weak topology and consider for the subset $\mathcal{Z}$ a convex set which includes an upper bound for some of the moments and some $L^{p}$-norm, with $p>1$, which yield the desired compactness in $\mathcal{Y}$.  As a consequence, with such approach a crucial point in the analysis is to determine uniform $L^{p}$-norm bounds for the self-similar evolution problem.  This last particular issue, if true, appear to be quite difficult to prove in the model \eqref{gs}, mainly because the lack of angular averaging in the 1-D interaction operator $\Q$ as opposed to higher dimensional interaction operators.  This problem is reminiscent of related 1-D interaction operators associated to coagulation--fragmentation problems, for instance in Smoluchowski equation, for which propagation of $L^{p}$ norms is hard to establish, see \cite{LM} for details.  Having this in mind, it appears to us more natural to work with measure solutions and considering then $\mathcal{Y}$ as a suitable space of real Borel measures endowed with the weak-$\star$ topology for which the compactness will be easier to establish.  Of course, the main difficulty will then be to determine that the fixed point provided by Theorem \ref{GPV}  is not the Dirac measure at zero (the trivial solution is easily discarded by mass conservation).  In fact, we will prove that this steady state is a $L^{1}$ function (see Theorem \ref{theo:meas-L1}).  Let us introduce some notations before entering in more details. 
\subsection{Notations} Let us introduce the set $\M_{s}(\R)$ as the Banach space of real Borel measures on $\R$ with finite total variation of order $s$ endowed with the norm $\|\cdot\|_{s}$ defined as
\begin{equation*}  
\|\mu\|_{s}:=\int_{\R}\la x \ra^{s} \,|\mu\,|(\d x) < \infty\,, \qquad \text{ with  } \quad \la x \ra:=\left(1+|x|^{2}\right)^{\frac{1}{2}} \qquad \forall x \in \R\,,
\end{equation*}
where the positive Borel measure $|\mu|$ is the total variation of $\mu$.  We also set
\begin{equation*}
\M^{+}_{s}(\R)=\{\mu \in \M_{s}(\R)\,;\,\mu \geq 0\}\,,
\end{equation*}
and denote by $\P(\R)$ the set of probability measures over $\R$.  For any $k \geq 0$, define the set
\begin{equation*}
\P_{k}(\R)=\left\{\mu \in \P(\R)\,;\,\int_{\R}|x|^{k}\mu(\d x) < \infty\right\}\,.
\end{equation*}
For any $\mu \in \P_{k}(\R)$ and any $0 \leq p \leq k$, let us introduce the $p$-moment $M_{p}\left(\mu\right):=\int_{\R}|x|^{p}\mu(\d x)$.  If $\mu$ is absolutely continuous with respect to the Lebesgue measure with density $f$, i.e. $\mu(\d x)=f(x)\d x$, we simply denote $M_{p}(f)=M_p(\mu)$ for any $p \geq 0$.  We also define, for any $k \geq 1$, the set
\begin{equation*}
\P_{k}^{0}(\R)=\left\{\mu \in \P_{k}(\R)\,;\,\int_{\R}x \mu(\d x)=0\right\}\,.
\end{equation*}
In the same way, we set $L^{1}_{k}(\R)=L^{1}(\R) \cap \M_{k}(\R)$, for any $k \geq 0$.
Moreover, we introduce the set $L^{\infty}_{-s}$ $(s \geq 0)$ of locally bounded  Borel functions $\varphi$ such that
\begin{equation*}
\|\varphi\|_{L^{\infty}_{-s}}:=\sup_{x \in \R} |\varphi(x)\,|\,\la x \ra^{-s}\, < \infty\,.
\end{equation*}
For any $p \geq 1$ and $\mu,\nu \in \P_{p}(\R)$, we recall the definition of the Wasserstein distance of order $p$, $W_{p}(\mu,\nu)$ between $\mu$ and $\nu$ 
by
\begin{equation*}
W_{p}(\mu,\nu)=\left(\inf_{\pi \in \Pi(\mu,\nu)}\int_{\R^{2}}|x-y|^{p}\pi(\d x,\d y)\right)^{\frac{1}{p}}\,,
\end{equation*} 
where $\Pi(\mu,\nu)$ denotes the set of all joint probability measures $\pi$ on $\R^{2}$ whose marginals are $\mu$ and $\nu$. For the peculiar case $p=1$, we shall address the \textit{first order Wasserstein distance} $W_{1}(\mu,\nu)$ as \textit{Kantorovich--Rubinstein distance}, denoted $d_{\mathrm{KR}}$, i.e. $d_{\mathrm{KR}}(\mu,\nu)=W_{1}(\mu,\nu)$. We refer to  \cite[Section 7]{vil} and \cite[Chapter 6]{vil2} for more details on Wasserstein distances.  The Kantorovich--Rubinstein duality asserts that
\begin{equation*}
d_{\mathrm{KR}}(\mu,\nu)=\sup_{\varphi \in \mathrm{Lip}_{1}(\R)}\int_{\R}\varphi(x)(\mu-\nu)(\d x)
\end{equation*}
where $\mathrm{Lip}_{1}(\R)$ denotes the set of Lipschitz functions $\varphi$ such that
\begin{equation*}
\|\varphi\|_{\mathrm{Lip}(\R)}=\sup_{x \neq y}\dfrac{|\varphi(x)-\varphi(y)|}{|x-y|} \leq 1\,.\end{equation*}
For a given $T >0$ and a given $k \geq 0$, we shall indicate as
$\C_{\text{weak}}([0,T],\P_{k}(\R))$
the set of continuous mappings from  $[0,T]$ to  $\P_{k}(\R)$ where the latter is endowed with the weak-$\star$ topology. 
\subsection{Collision operator and definition of measure solutions}
We extend the definition \eqref{eq:Qw} to nonnegative Borel measures; namely, given $\mu, \nu \in \M_{\gamma}^{+}(\R)$, let
\begin{equation} \label{eq:Qmu}
\langle \Q(\mu,\nu)\,;\,\varphi\rangle
:=\frac{1}{2}\int_{\R^{2}} \big|x-y\big|^{\gamma}\,\Delta \varphi(x,y)
\mu(\d x)\nu(\d  y)
\end{equation}
for any test function $\varphi \in \mathcal{C}(\R) \cap L^{\infty}_{-\gamma}(\R)$ where
\begin{equation*}
\Delta \varphi\big(x,y\big):=\varphi\big(\a x+\b y\big)+\varphi\big(\b x+\a y\big)-\varphi\big(x\big)-\varphi\big(y\big), \qquad (x,y) \in \R^{2}\,.
\end{equation*} 
A natural definition of measure solutions to \eqref{eq:1} is the following, (see \cite{Lu}).
\begin{defi}\label{defi:weak}
Let $\gamma > 0$, $\gmm=\max(\gamma,2)$, $\mu_{0} \in \M_{\gmm}^{+}(\R)$  and $(\mu_{t})_{t \geq 0} \subset \M_{\gmm}^{+}(\R)$ be given. We say that $(\mu_{t})_{t \geq 0}$ is a measure weak solution to \eqref{eq:1} associated to the initial datum $\mu_{0}$ if it satisfies
\begin{enumerate}
\item $\sup_{t \geq 0}\|\mu_{t}\|_{\gmm} \leq \|\mu_{0}\|_{\gmm}$,
\beq\label{eq:cons}
\int_{\R}\mu_{t}(\d x)=\int_{\R}\mu_{0}(\d x) \qquad \text{ and } \qquad \int_{\R} x\mu_{t}(\d x)=\int_{\R}x \mu_{0}(\d x)\qquad \forall t > 0\,.
\eeq
\item for any test-function $\varphi \in \C_{b}(\R):=\C(\R) \cap L^{\infty}(\R)$, the following hold
\begin{enumerate}[i)]
\item the mapping $t  \mapsto \langle \Q(\mu_{t},\mu_{t})\,;\,\varphi\rangle$ belongs to $\C\big([0,\infty)\big) \cap L^{1}_{\mathrm{loc}}\big([0,\infty)\big)$,  
\item for any $t \geq 0$ it holds 
\begin{equation}\label{eq:sol}
\int_{\R}\varphi(x)\mu_{t}(\d x)=\int_{\R}\varphi(x)\mu_{0}(\d x)  + \int_{0}^{t}\langle \Q\big(\mu_{\tau},\mu_{\tau}\big)\,;\,\varphi\rangle\,\d\tau\,.
\end{equation}
\end{enumerate}
\end{enumerate}
\end{defi}
\noindent
Notice that if $\mu_{t} \in \M_{\gmm}^{+}(\R)$ for any $t \geq 0$, then
\begin{equation*}
\int_{\R^{2}} \big|x-y\big|^{\gamma}\,\big|\Delta \varphi(x,y)\big|\,\mu_{t}(\d x) \mu_{t}(\d y) \leq 4\|\varphi\|_{\infty}\|\mu_{t}\|_{\gamma}^{2}< \infty
\end{equation*}
for any $\varphi \in \C_{b}(\R)$ and $t \geq 0$.  This shows that $\la \Q(\mu_{t},\mu_{t})\,;\,\varphi\ra$ is well-defined for any $t \geq 0$ and any $\varphi \in \C_{b}(\R)$.  Similarly, the notion of measure solution to \eqref{steady} is given in the following statement. 
\begin{defi}\label{defi:station}
A measure $\mub \in \mathcal{P}_{\max(\gamma,2)}^{0}(\R)$ is a solution to \eqref{steady} if 
\begin{equation}\label{meas_form}
-  \int_{\R} \xi \phi'(\xi)\mub(\d\xi) =  \frac{1}{2}\int_{\R^2} |\xi-\eta|^\gamma 
\big(\phi\left(\a \xi + \b \eta\right)+\phi\left(\a \eta+\b \xi\right)-\phi(\xi)- \phi(\eta)\big) \mub(\d\xi)\,\mub(\d\eta)
\end{equation}
for any $\phi\in\mathcal{C}^1_b (\R)$ where $\phi'$ stands for the derivative 
of $\phi$. 
\end{defi} 
Notice that, by assuming $\|\mub\|_{0}=1$, we naturally discard the trivial solution $G=0$ to \eqref{steady}.
\subsection{Strategy and main results}
 Fix $\gamma >0$, $\gmm=\max(\gamma,2)$. Thanks to the conservative properties \eqref{eq:cons}, we shall assume in the sequel, and without any loss of generality, that the initial datum $\mu_{0} \in \M_{\gmm}^{+}(\R)$ is such that
\begin{equation*}
\int_{\R}\mu_{0}(\d x)=1 \qquad \text{ and }  \qquad \int_{\R}x \mu_{0}(\d x)=0\,,
\end{equation*}
i.e. $\mu_{0} \in \P_{\gmm}^{0}(\R)$.  This implies that any weak measure solution $(\mu_{t})_{t\geq 0}$ to \eqref{eq:1} associated to $\mu_{0}$ is such that
\begin{equation*}
\mu_{t} \in \P_{\gmm}^{0}(\R) \qquad \forall t \geq 0\,.
\end{equation*} 
\begin{theo}\label{theo:cauchy} Let $\mu_{0} \in \P_{\gmm}^{0}(\R)$ be given initial datum.  Then, there exists a measure weak solution $(\mu_{t})_{t\geq 0}$ to \eqref{eq:1} associated to $\mu_{0}$ in the sense of Definition \ref{defi:weak}.  Moreover, 
\begin{align*}
\int_{\R}\mu_{t}(\d x)=\int_{\R}\mu_{0}(\d x)=1\,, &\qquad \int_{\R}x\mu_{t}(\d x)=\int_{\R}x \mu_{0}(\d x) =0\,, \\
\text{and }\;\;\int_{\R}|x|^{\gmm}\mu_{t}(\d x) &\leq \int_{\R}|x|^{\gmm}\mu_{0}(\d x) \qquad \forall t \geq 0\,.
\end{align*}
Moreover, if $\mu_0 \in \bigcap_{k\geq 0} \P^0_k(\R)$ then $(\mu_{t})_{t} \subset \bigcap_{k \geq 0}\P^{0}_{k}(\R)$. 
If additionally there exists $\varepsilon > 0$ such that
\begin{equation}\label{eq:exp}
\int_{\R} \exp(\varepsilon |x|^{\gamma})\mu_{0}(\d x) < \infty\,,
\end{equation}
then, such a measure weak solution is unique.  Furthermore, if $\mu_{0}$ is absolutely continuous with respect to the Lebesgue measure, i.e. $\mu_{0}(\d x)=f_{0}(x)\d x$ with $f_{0} \in L^{1}_{\gmm}(\R)$ then $\mu_{t}$ is absolutely continuous with respect to the Lebesgue measure for any $t \geq 0$.  That is, there exists $(f_{t})_{t\geq 0} \subset L^{1}_{\gmm}(\R)$ such that $\mu_{t}(\d x)=f_{t}(x)\d x$ for any $t \geq 0$.
\end{theo}
We prove Theorem \ref{theo:cauchy} following a strategy introduced in \cite{Lu}-\cite{FM} for the case of Boltzmann equation with hard potentials (with or without cut-off). The program consists essentially in the following steps. (1) Establish \textit{a priori} estimates for measure weak solutions to \eqref{eq:1} concerning the creation and propagation of algebraic moments, (2) construct measure weak solutions to \eqref{eq:1} by approximation of $L^{1}$-solutions.  Step (1) helps proving that such approximating sequence converges in the  weak-$\star$ topology.  Finally, (3) for the uniqueness of measure weak solution we adopt a strategy developed in \cite{FM} and based upon suitable Log-Lipschitz estimates for the Kantorovich--Rubinstein distance between two solutions of \eqref{eq:1}.\\

As far as \textbf{Question 1} is concerned, we establish the optimal decay of the moments of the solutions to \eqref{eq:1} by a suitable comparison of ODEs.  Such techniques are natural and have been applied to the study of Haff's law for 3-D granular Boltzmann equation in \cite{AL1}.  The main difficulty is to provide optimal lower bounds for the moments, see Propositions \ref{l2} \& \ref{p2}.  Essentially, we obtain that
\begin{equation*}
m_{k}(t)=\int_{\R}|x|^{k}\,\mu_{t}(\d x)\approx C_{k}\,t^{-\frac{k}{\gamma}} \quad \text{ as }\; t \to \infty\,,
\end{equation*}
see Theorem \ref{theo:decmom} for a more precise statement. Such a decay  immediately translates into convergence of $\mu_{t}$ towards $\delta_{0}$ in the Wasserstein topology.\\

Regarding \textbf{Question 2}, once the fixed point Theorem \ref{GPV} is at hand, the key step is to engineer a suitable \textit{stable} compact set $\mathcal{Z}$.  Compactness is easily achieved in $\mathcal{Y}=\P_{\max(\gamma,2)}(\mathbb{R})$ endowed with weak-$\star$ topolgy; only uniformly boundedness of some moments suffices.  However, $\mathcal{Y}$ must overrule the possibility that the fixed point will be a plain Dirac mass located at zero.  This is closely related to the sharp lower bound found  for the moments in \textbf{Question 1}.  In such a situation, a serie of simple observations on the regularity of the solution to \eqref{meas_form} proves that such a steady state is actually a $L^{1}$-function. Namely, one of the most important step in our strategy is the following observation:
\begin{theo}\label{theo:meas-L1}
Any steady measure solution $\mub \in \P_{\max(\gamma,2)}^{0}(\R)$ to \eqref{steady} such that
\begin{equation}\label{eq:mom}
\mathbf{m}_{\gamma}:=\int_{\R} |\xi|^{\gamma}\mub(\d \xi) > 0
\end{equation}
is absolutely continuous with respect to the Lebesgue measure over $\R$, i.e. there exists some nonnegative $G \in L^{1}_{\max(\gamma,2)}(\R)$ such that
\begin{equation*}
\mub(\d\xi)=G(\xi)\d\xi\,.
\end{equation*}
\end{theo}
In other words, any solution to \eqref{steady} lying in $\P_{\max(\gamma,2)}^{0}(\R)$ different from a Dirac mass must be a regular measure.  This leads to our main result:
\begin{theo}\label{existence}
For any $\gamma > 0$, there exists  $G \in L^{1}_{\max(\gamma,2)}(\R)$ which is a steady solution to \eqref{steady} in the weak sense. 
\end{theo}
\section{Cauchy theory}\label{sec:cauchy}
We are first concerned with the Cauchy theory for problem \eqref{eq:1} and we begin with studying \emph{a priori} estimates for weak measure solutions to \eqref{eq:1}. Let us fix $\gamma > 0$ and set $\gmm=\max(\gamma,2).$

\subsection{\textit{A priori} estimates on moments}
We first state the following general properties of weak measure solutions to \eqref{eq:1} which have sufficiently enough bounded moments. 
\begin{prop}\label{prop:a priori}
Let $\mu_{0} \in \M_{\gmm}^{+}(\R)$ and let $(\mu_{t})_{t\geq 0}$ be any weak measure solution to \eqref{eq:1} associated to $\mu_{0}$.  Given $k \geq 0$, assume there exists $p > k+\gamma $ such that 
\begin{equation}\label{eq:hypmo}
\sup_{\delta < t < T}\|\mu_{t}\|_{p} < \infty\,, \qquad \forall\;T > \delta >0\,.
\end{equation}
Then, the following hold:
\begin{enumerate}
\item For  any $\varphi \in L_{-k}^{\infty}(\R) \cap \C(\R)$, the mapping $t \geq 0 \mapsto \langle \Q(\mu_{t},\mu_{t})\,;\,\varphi\rangle$ is continuous in $(0,\infty)$. 
\item For  any  $\varphi \in L_{-k}^{\infty}(\R) \cap \C(\R)$ it holds that
\begin{equation}\label{eq:deriv}
\dfrac{\d}{\d t} \int_{\R}\varphi(x)\mu_{t}(\d x)=\langle \Q(\mu_{t},\mu_{t})\,;\,\varphi\rangle \qquad \forall\, t \geq 0\,.
\end{equation}
\end{enumerate}
\end{prop}
The proof of Proposition \ref{prop:a priori} will need the following  preliminary lemma (see \cite[Proposition 2.2]{Lu} for a complete proof).
\begin{lem}\label{lem_lu}
Let $(\mu^n)_{n\in\N}$ be a sequence from $ \M_{\gmm}^{+}(\R)$  that converges weakly-$\star$ to some $\mu \in \M_{\gmm}^{+}(\R)$. We assume that  for some $p>0$  
it holds
\begin{equation*}
\sup_{n\in\N}\|\mu^n\|_{p} < \infty\,.
\end{equation*}
Then, for any $\psi\in\C(\R^2)$ satisfying $\displaystyle \lim_{|x|+|y|\to +\infty} \frac{\psi(x,y)}{\la x \ra^{p}+\la y \ra^{p}} =0$, one has
\begin{equation*} 
\lim_{n\to +\infty} \int_{\R^2}\psi(x,y)\,\mu^n(\d x)\,\mu^n(\d y) 
=\int_{\R^2} \psi(x,y)\, \mu(\d x)\, \mu(\d y).
\end{equation*}
\end{lem}
\begin{proof}[Proof of Proposition \ref{prop:a priori}] Let $k \geq 0$  and $\varphi \in L^{\infty}_{-k}(\R) \cap \C(\R)$ be given.   Choose $\chi \in \C^{\infty}(\R)$ such that $\|\chi\|_{\infty} \leq 1$ with $\chi(x)=0$ if $|x| \geq 2$ and $\chi(x)=1$ for $|x| \leq 1$ and set for any $n \in \mathbb{N}_*$, $\varphi_{n}(x)=\varphi(x)\chi\left(\frac{x}{n}\right)$.  It follows that $\varphi_{n} \in \C_{c}(\R) \subset \C_{b}(\R)$ for any $n$ with $\varphi_{n}(x) \to \varphi(x)$ for any $x \in \R$ as $n \to \infty$. Consequently, $\Delta \varphi_{n}(x,y) \to \Delta\varphi(x,y)$ for any $(x,y) \in \R^{2}$ as $n \to \infty$.  Now, since $\varphi_{n}\in \C_{b}(\R)$ for any $n \in \mathbb{N}$, one deduces from \eqref{eq:sol} that
\begin{equation*}
\int_{\R}\varphi_{n}(x)\mu_{t_2}(\d x)=\int_{\R}\varphi_{n}(x)\mu_{t_1}(\d x)+\int_{t_1}^{t_2}\la \Q(\mu_{\tau},\mu_{\tau})\,;\,\varphi_{n}\ra \,\d\tau \qquad \forall\; t_2 > t_1 > 0, \quad \forall\; n \geq 1\,.
\end{equation*} 
Notice that $\|\varphi_{n}\|_{L^{\infty}_{-k}} \leq \|\varphi\|_{L^{\infty}_{-k}} < \infty$. Thus, there is $C > 0$ such that for any $n\in\N_*$,  any $(x,y)\in\R^2$,
\begin{equation}\label{delta}
|\Delta \varphi(x,y)| \leq C\left(\la x \ra^{k}+ \la y \ra^{k}\right) \qquad 
\mbox{ and } \qquad
|\Delta \varphi_n(x,y)| \leq C\left(\la x \ra^{k}+ \la y \ra^{k}\right)\,,
\end{equation}
where the constant $C$ depends only on $\varphi$ (and $\a$).  Using the dominated convergence theorem together with \eqref{eq:hypmo}, one deduces that
\begin{align*}
\int_{t_1}^{t_2}\la \Q(\mu_{\tau},&\mu_{\tau})\,;\,\varphi_{n}\ra \,\d\tau=\frac{1}{2}\int_{t_1}^{t_2}\d\tau\int_{\R^{2}}|x-y|^{\gamma}\Delta\varphi_{n}(x,y)\mu_{\tau}(\d x)\mu_{\tau}( \d y)\\
&\underset{n \to \infty}{\longrightarrow} \frac{1}{2}\int_{t_1}^{t_2}\d\tau\int_{\R^{2}}|x-y|^{\gamma}\Delta\varphi(x,y)\mu_{\tau}(\d x)\mu_{\tau}( \d y)= \int_{t_1}^{t_2}\la \Q(\mu_{\tau},\mu_{\tau})\,;\,\varphi\ra \,\d\tau\,,  
\end{align*}
so, the identity
\begin{equation}\label{eq:st}
\int_{\R}\varphi(x)\mu_{t_2}(\d x)=\int_{\R}\varphi(x)\mu_{t_1}(\d x)+\int_{t_1}^{t_2}\la \Q(\mu_{\tau},\mu_{\tau})\,;\,\varphi\ra \,\d\tau \qquad \forall\; t_2 > t_1 > 0\,,
\end{equation}
holds true. It follows from \eqref{delta} that
\begin{equation*} 
\bigg|\la \Q(\mu_{\tau},\mu_{\tau})\,;\,\varphi\ra\bigg| \leq C \int_{\R^{2}}|x-y|^{\gamma}\left(\la x \ra^{k}+ \la y \ra^{k}\right)\mu_{\tau}(\d x)\mu_{\tau}(\d y) \leq 2C\|\mu_{\tau}\|_{k+\gamma}^{2}\qquad \forall \tau >0\,.
\end{equation*}
Combining this with \eqref{eq:st}, one sees that, for any $T > \delta > 0$, 
\begin{equation*}
\bigg|\int_{\R}\varphi(x)\mu_{t_2}(\d x)-\int_{\R}\varphi(x)\mu_{t_1}(\d x)\bigg| \leq 2\,C\,|t_2-t_1|\,\sup_{\delta \leq \tau \leq T}\|\mu_{\tau}\|_{k+\gamma}^{2} \qquad \forall\; t_{1},\,t_{2} \in [\delta,T]\,.
\end{equation*}
In particular, under assumption \eqref{eq:hypmo}, the mapping $t \mapsto \ds\int_{\R}\varphi(x)\mu_{t}(\d x)$ is continuous over $(0,\infty)$. This shows that, for any $t > 0$ and sequence $(t_{n})_{n} \subset [t/2,3t/2]$ with $\lim_{n}t_{n}=t$, the sequence $(\mu_{t_{n}})_{n}$ converges weakly-$\star$ towards $\mu_{t}$.  In addition, since $\varphi \in L^\infty_{-k}(\R)$ and $p > k+\gamma$ one concludes that
\begin{equation*}
\lim_{|x|+|y| \to \infty} \dfrac{|x-y|^{\gamma}\Delta \varphi(x,y)}{\la x \ra^{p}+\la y \ra^{p}}=0\,.
\end{equation*}
Thus, it readily follows from Lemma \ref{lem_lu} that 
\begin{equation}\label{eq:tnt}
\lim_{n} \int_{\R^{2}}|x-y|^{\gamma}\Delta \varphi(x,y)\mu_{t_{n}}(\d x)\mu_{t_{n}}(\d y)=\int_{\R^{2}}|x-y|^{\gamma}\,\Delta \varphi(x,y)\,\mu_{t}(\d x)\mu_{t}(\d y)\,.
\end{equation}
Henceforth, the mapping $t \mapsto \la \Q(\mu_{t},\mu_{t}),\varphi\ra$ is continuous over $(0,\infty)$ proving point \textit{(1)}.  Point \textit{(2)} follows then directly from \eqref{eq:st}.
\end{proof}
Moments estimates of the collision operator are given by the following: 
\begin{prop}
Let $\mu \in \P_{k+\gamma}^0(\R)$  with $k \geq 2$. Then
\begin{subequations}
\begin{equation}  \label{eq:Qmuk}
\langle \Q(\mu,\mu)\,;\,|\cdot|^{k}\rangle \leq -\frac{1}{2}\left(1-\a^{k}-\b^{k}\right) M_{k+\gamma}(\mu) \leq 0\,.
\end{equation}
In particular,
\begin{equation}\label{eq:Qmuk1}
\langle \Q(\mu,\mu)\,;\,|\cdot|^{k}\rangle  \leq -\frac{1}{2}\left(1-\a^{k}-\b^{k}\right)M_{k}\left(\mu\right)^{1+\frac{\gamma}{k}}\,,
\end{equation}
and, if $k >2$
\begin{equation} \label{eq:Qmuk2}
\langle \Q(\mu,\mu)\,;\,|\cdot|^{k}\rangle 
\leq -\frac{1}{2}\left(1-\a^{k}-\b^{k}\right)M_{2}(\mu)^{-\frac{\gamma}{k-2}}\,M_{k}(\mu)^{1+\frac{\gamma}{k-2}}.\end{equation}
\end{subequations}
\end{prop}
\begin{proof} 
We apply the weak form \eqref{eq:Qmu} to $\varphi(x)=|x|^{k}$.  Using the elementary inequality
\begin{equation}\label{elem}
|x|^{k}+|y|^{k}-\left|\a x+\b y\right|^{k} - \left|\a y +\b x\right|^{k}  \geq \left(1-\a^{k}-\b^{k}\right)\,|x-y|^{k}
\end{equation}
valid for any $(x,y) \in \R^{2}$ and $k \geq 2$ (with equality sign whenever $k=2$), and noticing that $1-\a^{k}-\b^{k}$ is nonnegative for any $k \geq 2$ and any $\a \in (0,1)$ we have 
\begin{equation}\label{eq:Qk}
\langle \Q(\mu,\mu)\,;\,|\cdot|^{k}\rangle \leq -\frac{1}{2}\left(1-\a^{k}-\b^{k}\right)\int_{\R^{2}}|x-y|^{\gamma+k}\mu(\d x)\mu(\d  y)\,.
\end{equation}
Since $\mu \in \P_{k+\gamma}^{0}(\R)$ and the mapping $[0,\infty)\ni r \mapsto r^{\gamma+k}$ is convex, one deduces from Jensen's inequality that
\begin{equation*}
\int_{\R^{2}}|x-y|^{\gamma+k}\mu(\d x)\mu(\d y) \geq \int_{\R}|x|^{k+\gamma}\mu(\d x)\,,
\end{equation*}
from which inequality \eqref{eq:Qk} yields  \eqref{eq:Qmuk}.  
Using again Jensen's inequality we get also that $M_{k+\gamma}(\mu) \geq M_{k}(\mu)^{1+\frac{\gamma}{k}}$ which proves \eqref{eq:Qmuk1}. Finally, according to H\"{o}lder inequality
\begin{equation*}
M_{k+\gamma}(\mu) \geq M_{k}(\mu)^{1+\frac{\gamma}{k-2}}\,M_{2}(\mu)^{-\frac{\gamma}{k-2}} \qquad \forall\; k >2\,,
\end{equation*}
and \eqref{eq:Qmuk2} is deduced from \eqref{eq:Qmuk}.
\end{proof}
One deduces from the above estimate suitable \emph{upper bounds} for the moments of solution to \eqref{eq:1}:
\begin{prop}\label{prop:decrease}
Let $\mu_{0} \in \P_{\gmm}^{0}(\R)$ be a given initial datum and $(\mu_{t})_{t \geq 0}$ be a measure weak solution to \eqref{eq:1} associated to $\mu_{0}$. Then, the following holds
\begin{enumerate}
\item If  $(\mu_{t})_{t\geq 0}\subset \P_{p}^{0}(\R)$ with $p>k+\gamma$ for some $k \geq 2$, then, $t \geq 0 \mapsto M_{k}(\mu_{t})$ is decreasing with 
\begin{equation}\label{eq:decay}
 M_{k}\left(\mu_{t}\right)   \leq M_{k}\left(\mu_{0}\right)\Big(1+\frac{\gamma}{2k}\big(1-\a^{k}-\b^{k}\big)M_{k}(\mu_{0})^{\frac{\gamma}{k}}\,t \Big)^{-\frac{k}{\gamma}} \qquad \forall \; t \geq 0\,.
 \end{equation} 
In particular, $ \sup_{t \geq 0}M_{k}(\mu_{t}) < \infty$.
\item If for some $k>2$ and $p>k+\gamma$,  $\sup_{t \geq \delta}\|\mu_{t}\|_{p} < \infty$ for any $\delta > 0$, then
\begin{equation}\label{eq:decay2}
M_{k}(\mu_{t}) \leq  C_{k}(\gamma)\,\|\mu_{0}\|_{\gmm}\,\min\big\{ t^{-\frac{k-2}{\gamma}},t^{-\frac{k}{\gamma}}\big\} \qquad \forall\; t > 0\,,
\end{equation}
where $C_{k}(\gamma) > 0$ is a positive constant depending only on $k > 2$, $\gamma$ and $\a$.
\end{enumerate}
\end{prop}
\begin{proof}
For the proof of \eqref{eq:decay}, apply \eqref{eq:deriv} with $\varphi(x)=|x|^{k}$ and note that using \eqref{eq:Qmuk1}
\begin{equation*}
\dfrac{\d}{\d t}M_{k}(\mu_{t}) \leq -\frac{1}{2}\left(1-\a^{k}-\b^{k}\right)\,M_{k}\left(\mu_{t}\right)^{1+\frac{\gamma}{k}} \qquad \forall t \geq 0\,.
\end{equation*}
This directly implies the point (1) of Proposition \ref{prop:decrease}. To prove estimate \eqref{eq:decay2} for short time observe that applying \eqref{eq:deriv} to $\varphi(x)=|x|^{k}$ and using \eqref{eq:Qmuk2}, one gets
\begin{equation*}
\dfrac{\d}{\d t}M_{k}(\mu_{t}) \leq -\frac{1}{2}\left(1-\a^{k}-\b^{k}\right)M_{2}(\mu_{t})^{-\frac{\gamma}{k-2}}\,M_{k}(\mu_{t})^{1+\frac{\gamma}{k-2}} \qquad \forall\; t > 0\,,
\end{equation*}
and, since by definition of measure weak solution it holds $M_{2}(\mu_{t}) \leq \|\mu_{0}\|_{\gmm}$ for any $t \geq 0$, one deduces that
\begin{equation*}
\dfrac{\d}{\d t}M_{k}(\mu_{t}) \leq -\frac{1}{2}\left(1-\a^{k}-\b^{k}\right)\|\mu_{0}\|_{\gmm}^{-\frac{\gamma}{k-2}}\,M_{k}(\mu_{t})^{1+\frac{\gamma}{k-2}} \qquad \forall\; t > 0\,.
\end{equation*}
This inequality leads to estimate \eqref{eq:decay2} with constant $C_{k}(\gamma)=\left(\frac{\gamma}{2(k-2)}\left(1-\a^{k}-\b^{k}\right)\right)^{-\frac{k-2}{\gamma}}$.  The long time decay follows applying \eqref{eq:decay} for $t\geq\delta$.
\end{proof}
\begin{nb}\label{nb:gene-rho}
It is not difficult to prove that the conclusions of the previous two propositions hold for general measure $\mu_{0} \in \M^{+}_{\gmm}(\R)$ with
$$\|\mu_{0}\|_{0}=\varrho\neq 0\;\; \text{and}\;\; \ds\int_{\R}x \mu_{0}(\d x)=0\,.$$
In such a case, the constant $C_{k}(\gamma)$ depends also continuously on $\varrho$.
\end{nb}
\noindent
Introduce the class $\P_{\exp,\gamma}(\R)$ of probability measures with exponential tails of order $\gamma$,
\begin{equation}\label{eq:pexp}
\P_{\exp,\gamma}(\R)=\left\{\mu \in \P(\R)\;\,;\,\exists \varepsilon > 0 \text{ such that } \int_{\R}\exp(\varepsilon |x|^{\gamma}) \mu(\d x) < \infty\right\}\,.
\end{equation}
We have the following:
\begin{theo}\label{theo:expo} 
Let $\mu_{0}\in \P_{\gmm}^0(\R)$ be an initial datum and $(\mu_{t})_{t}$ be a measure weak solution to \eqref{eq:1} associated to $\mu_{0}$.  If $\mu_0 \in \P_{\exp,\gamma}(\R)$ and if $(\mu_t)_t \subset \bigcap_{k\geq 0} \P^0_k(\R)$, then, there exists $\alpha > 0$ and $C > 0$ such that  
\begin{equation} \label{eq:expon2}
\sup_{t \geq 0}\int_{\R}\exp(\alpha |x|^{\gamma})\mu_{t}(\d x) \leq C\,.
\end{equation}
\end{theo}
\begin{proof}
Since  $\mu_{0}\in \P_{\exp,\gamma}(\R)$, there exists $\alpha>0$ and $C_0>0$ such that
\begin{equation*}
\int_\R \exp(\alpha |x|^\gamma)\mu_0(\d x) \leq C_0\,.
\end{equation*}
Let us denote by $p_0$ the integer such that $\gamma\, p_0\geq \gmm$ and $\gamma\,(p_0-1)<\gmm$.  Thus, for $0\leq p\leq p_0-1$ and $t\geq 0$,
\begin{equation*} 
M_{\gamma p } (\mu_t)\leq \|\mu_t\|_{\gmm}\leq \|\mu_0\|_{\gmm}\,.
\end{equation*}
For $p\geq p_0$ and $t\geq 0$, one deduces from \eqref{eq:Qmuk} that
\begin{equation*} 
M_{\gamma p } (\mu_t)\leq M_{\gamma p } (\mu_0)\,.
\end{equation*}
Thus, for any $t\geq 0$ and any $n>p_0$
\begin{equation*}
\sum_{p=0}^n M_{\gamma p } (\mu_t) \frac{\alpha^p}{p!} \leq 
\|\mu_0\|_{\gmm} \sum_{p=0}^{p_0-1} \frac{\alpha^p}{p!} + \sum_{p=p_0}^n M_{\gamma p } (\mu_0) \frac{\alpha^p}{p!} \leq \|\mu_0\|_{\gmm} \exp(\alpha) +C_0\,.
\end{equation*}
Letting $n\to +\infty$ we get \eqref{eq:expon2}.
\end{proof}
\subsection{Cauchy theory}\label{sec:steps} The main ingredients of the proof of Theorem \ref{theo:cauchy} are Propositions \ref{prop:existence} and \ref{prop:stab}.  The details of the proof can be found in Appendix \ref{sec:cauchyA}. Namely, by studying first the Cauchy problem for $L^{1}$ initial data and then introducing suitable approximation we can construct weak measure solutions to \eqref{eq:1} leading the following existence result:
\begin{prop}\label{prop:existence}
For any $\mu_{0} \in \P_{\gmm}^{0}(\R)$, $\mu_{0} \neq \delta_{0}$, there exists a measure weak solution $(\mu_{t})_{t \geq 0} \subset \P_{\gmm}^{0}(\R)$ associated to $\mu_{0}$ and such that
\begin{equation*}
\sup_{t \geq 0}\|\mu_{t}\|_{\gmm} \leq \|\mu_{0}\|_{\gmm} \qquad \text{ and }  \qquad \sup_{t \geq t_{0}}\|\mu_{t}\|_{s} < \infty \qquad \forall\, t_{0} > 0,\, s > \gmm\,.
\end{equation*}
Moreover, if $\mu_{0} \in \bigcap_{k\geq 0} \P^0_k(\R)$, then $(\mu_t)_{t\geq 0} \subset \bigcap_{k\geq 0} \P^0_k(\R)$.
\end{prop}

To achieve the proof of Theorem \ref{theo:cauchy}, it remains only to prove the uniqueness of the solution. It seems difficult here to adapt the strategy of \cite{Lu} for which the conservation of energy played a crucial role. For this reason, we rather follow the approach of \cite{FM} which requires the exponential tail estimate \eqref{eq:exp}. The main step towards uniqueness is the  following \emph{Log-Lipschitz estimate} for the Kantorovich-Rubinstein distance.

\begin{prop}\label{prop:stab}
Let $\mu_0$ and $\nu_{0}$ be two probability measures in $\P_{2}^{0}(\R)$ satisfying  \eqref{eq:exp} i.e., there exists $\varepsilon > 0$ such that
\begin{equation*}
\int_{\R}\exp\left(\varepsilon \,|x|^{\gamma}\right)\mu_{0}(\d x)+ \int_{\R}\exp\left(\varepsilon \,|x|^{\gamma}\right)\nu_{0}(\d x) < \infty\,.
\end{equation*}
Let then $(\mu_{t})_{t \geq 0}$ and $(\nu_{t})_{t \geq 0}$ be two measure weak solutions to \eqref{eq:1} associated respectively to the initial data $\mu_{0}$ and $\nu_{0}$.  There exists $K_{\varepsilon}  > 0$ such that, for any $T > 0$,  
\begin{align}\label{eq:Wt}
d_{\mathrm{KR}}&(\mu_{t},\nu_{t}) \leq d_{\mathrm{KR}}(\mu_{0},\nu_{0}) \\
&+ K_{\varepsilon}\,C_{T}(\varepsilon) \int_{0}^{t} d_{\mathrm{KR}}(\mu_{s},\nu_{s})\Big(1+\big|\log d_{\mathrm{KR}}(\mu_{s},\nu_{s})\big|\Big)\d s \qquad \forall\, t \in [0,T]\,, \nonumber
\end{align}
where $C_{T}(\varepsilon)=\ds\sup_{t \in [0,T]}\int_{\R}\exp\left(\varepsilon \,|x|^{\gamma}\right)\left(\mu_{t}+\nu_{t}\right)(\d x)<\infty$.
\end{prop}
\begin{proof}[Proof of Theorem \ref{theo:cauchy}] Given Proposition \ref{prop:existence}, to prove Theorem \ref{theo:cauchy} it suffices to show the uniqueness of measure weak solutions to \eqref{eq:1}.  Let $\mu_0$ be a probability measure in $\P_{\gmm}^{0}(\R)$ satisfying \eqref{eq:exp} and let $(\mu_{t})_{t \geq 0}$ and $(\nu_{t})_{t \geq 0}$  be two measure weak solutions to \eqref{eq:1} associated to $\mu_{0}$.  From Proposition \ref{prop:stab}, given $T > 0$ there exists a finite positive constant $K_{T}$ such that
\begin{equation*}
d_{\mathrm{KR}}(\mu_{t},\nu_{t}) \leq K_T \int_{0}^{t} \mathbf{\Phi}(d_{\mathrm{KR}}(\mu_{s},\nu_{s})) \d s \qquad \forall t \in [0,T]\,,
\end{equation*}
with $\mathbf{\Phi}(r)=r\left(1+|\log r|\right)$ for any $r > 0$.  Since $\mathbf{\Phi}$ satisfies the so-called \textit{Osgood condition}
\begin{equation}\label{eq:osgood}
\int_{0}^{1}\dfrac{\d r}{\mathbf{\Phi}(r)}=\infty\,,
\end{equation} 
then, a nonlinear version of Gronwall Lemma (see for instance \cite[Lemma 3.4, p. 125]{chemin}) asserts that $d_{\mathrm{KR}}(\mu_{t},\nu_{t})=0$ for any $t \in [0,T]$.   Since $T >0$ is arbitrary, this proves the uniqueness.
\end{proof}
The existence and uniqueness of weak measure solution to \eqref{eq:1} allows to define a semiflow $(\mathcal{S}_{t})_{t \geq 0}$  on $\P_{\exp,\gamma}(\R)$, (recall \eqref{eq:pexp}). Namely, Theorem \ref{theo:cauchy} together with Theorem \ref{theo:expo} assert  that, for any $\mu_{0} \in \P_{\exp,\gamma}(\R) \cap \P^{0}(\R)$, there exists a unique weak measure solution $(\mu_{t})_{t \geq 0}$ to \eqref{eq:1} with $\mu_{t} \in \P_{\exp,\gamma}(\R)$ for any $t \geq 0$ and we shall denote
\begin{equation*}
\mu_{t}:=\mathcal{S}_{t}(\mu_{0}) \qquad \forall t \geq 0\,.
\end{equation*}
Then, the semiflow $\mathcal{S}_{t}$ is a well-defined nonlinear mapping from $\P_{\exp,\gamma}(\R) \cap \P^{0}(\R)$ into itself.  Moreover, by definition of weak solution, the mapping $t \mapsto \mathcal{S}_{t}(\mu_{0})$ belongs to $\C_{\text{weak}}([0,\infty),\P_{2}(\R))$.  One has the following weak continuity result for the semiflow.
\begin{prop}\label{prop:cont}
The semiflow $(\mathcal{S}_{t})_{t \geq 0}$ is weakly continuous on $\P_{1}(\R)$ in the following sense.  Let $(\mu_{n})_{n} \in \P_{\gmm}^{0}(\R)$ be a sequence such that there exists $\varepsilon > 0$ satisfying 
\begin{equation}\label{eq:uniformdelta}
\sup_{n \in \mathbb{N}}\int_{\R}\exp(\varepsilon |x|^{\gamma})\mu_{n}(\d x) < \infty\,.
\end{equation} 
If $(\mu_{n})_{n}$ converges to some $\mu \in \P_{\gmm}^{0}(\R)$ in the weak-$\star$ topology, then for any $t \geq 0$,
\begin{equation*} 
\mathcal{S}_{t}(\mu_{n}) \longrightarrow \mathcal{S}_{t}(\mu) \qquad \text{ in the weak-$\star$ topology as $n \to \infty$}\,.
\end{equation*}
\end{prop}
\begin{proof}
The proof is based upon the stability result established in Proposition \ref{prop:stab}. Namely, because $(\mu_{n})_{n}$ converges in the weak-$\star$ topology to $\mu \in \P_{\gmm}^{0}(\R)$, one deduces from  \eqref{eq:uniformdelta} that
\begin{equation*} 
\int_{\R}\exp(\varepsilon |x|^{\gamma}) \mu(\d x) < \infty,
\end{equation*}
which means that all $\mu_{n}$ and $\mu$ share the same exponential tail estimate with \emph{some common} $\varepsilon > 0$. Then, for any $T > 0$ one deduces from \eqref{eq:Wt} that
\begin{align*}
d_{\mathrm{KR}}(\mathcal{S}_{t}(\mu_{n})\,,\,&\mathcal{S}_{t}(\mu)) \leq d_{\mathrm{KR}}(\mu_{n}\,,\,\mu) \\
&+ K(\varepsilon)\,C_{T}(\varepsilon)\int_{0}^{t} \mathbf{\Phi}\left(d_{\mathrm{KR}}(\mathcal{S}_{s}(\mu_{n})\,,\,\mathcal{S}_{s}(\mu)\right) \d s \qquad \forall\, t \in [0,T],\;n \in \mathbb{N}\,,
\end{align*}
for some universal positive constant $K(\varepsilon)$ and
\begin{equation*} 
C_{T}(\varepsilon)=\sup_{n \in \mathbb{N}}\sup_{t \in [0,T]}\int_{\R}\exp(\varepsilon |x|^{\gamma})\left(\mathcal{S}_{t}(\mu_{n})+\mathcal{S}_{t}(\mu)\right)(\d x) < \infty
\end{equation*}
according to Theorem \ref{theo:expo}. Here above, $\mathbf{\Phi}(r)=r\left(1+|\log r|\right)$ satisfies the Osgood condition \eqref{eq:osgood}, thus, using again a nonlinear version of Gronwall Lemma \cite[Lemma 3.4]{chemin} we deduce from this estimate that
\begin{equation*}
\mathbf{\Psi}\big(d_{\mathrm{KR}}\left(\mu_{n}\,,\,\mu\right)\big)-\mathbf{\Psi}\big(d_{\mathrm{KR}}\left(\mathcal{S}_{t}(\mu_{n})\,,\,\mathcal{S}_{t}(\mu)\right)\big) \leq K(\varepsilon)\,C_{T}(\varepsilon)\,t \qquad \forall\, t \in [0,T],\; n \in \mathbb{N}
\end{equation*}
where
\begin{equation*}
\mathbf{\Psi}(x)=\int_{x}^{1}\dfrac{\d r}{\mathbf{\Phi}(r)} \qquad \forall\, x \geq 0\,.
\end{equation*}
Taking now the limit $n \to \infty$, since $d_{\mathrm{KR}}$ metrizes the weak-$\star$ topology of $\P_{1}(\R)$ it follows that $\lim_{n\rightarrow\infty}d_{\mathrm{KR}}(\mu_{n},\mu)=0$.  Furthermore, recalling that $\mathbf{\Psi}(0)=\infty$ one concludes that
\begin{equation*} 
\lim_{n \to \infty}\mathbf{\Psi}\big(d_{\mathrm{KR}}\big(\mathcal{S}_{t}(\mu_{n})\,,\,\mathcal{S}_{t}(\mu)\big)\big) =\infty\,.
\end{equation*}
That is, $\lim_{n \to \infty}d_{\mathrm{KR}}\big(\mathcal{S}_{t}(\mu_{n})\,,\,\mathcal{S}_{t}(\mu)\big)=0$ which proves the result.
\end{proof}
\section{Optimal decay of the moments} 
We now prove that the upper bounds obtained for the moments of solutions to \eqref{eq:1} in Proposition \ref{prop:decrease} are actually optimal. 

\subsection{Lower bounds for moments}

We begin with the case $\gamma \in (0,1]$:
\begin{prop}\label{prop:even}
Fix $\gamma \in (0,1]$ and let $\mu_{0} \in \P_{2}^{0}(\R)$ be an initial datum and $(\mu_{t})_{t\geq 0}$ be a measure weak solution to \eqref{eq:1} associated to $\mu_0$.  Then, there exists $\mathcal{K}_{\gamma} > 0$ depending only on $\a$ and $\gamma$ such that
\begin{equation}\label{eq:Mg}
M_{\gamma}(\mu_{t}) \geq \frac{M_{\gamma}(\mu_{0})}{1+\mathcal{K}_{\gamma} M_{\gamma}(\mu_{0})t} \qquad \forall\, t \geq 0\,.
\end{equation}
Thus,
\begin{equation}\label{eq:lowM2}
M_{2}(\mu_{t}) \geq   M_{\gamma}(\mu_{0})^{\frac{2}{\gamma}} \bigg(1+\mathcal{K}_{\gamma}\,M_{\gamma}(\mu_{0}) t\bigg)^{-\frac{2}{\gamma}} \qquad \forall\, t \geq 0\,.
\end{equation}
\end{prop}
\begin{proof} For $\gamma \in (0,1]$ the following elementary inequality holds
\begin{equation}\label{eq:mga}
B_{\gamma}(x,y)=\Big(\left|\a x + \b y \right|^{\gamma} + \left|\a y  + \b x \right|^{\gamma}- |x|^{\gamma}- |y|^{\gamma}\Big)|x-y|^{\gamma} \geq -C_{\gamma}|x|^{\gamma}\,|y|^{\gamma} \quad \forall\, x,y \in \R\,,
\end{equation}
for some positive constant $C_{\gamma} >0$ explicit depending only on $\a$ and $\gamma$.  Using \eqref{eq:deriv} and the weak form \eqref{eq:Qmu} with $\varphi(x)=|x|^{\gamma} \in L^{\infty}_{-2}(\R) \cap \C(\R)$, we get that
\begin{equation*}
\dfrac{\d}{\d t}M_{\gamma}(\mu_{t})=\frac{1}{2}\int_{\R^{2}}B_{\gamma}(x,y)\mu_{t}(\d x)\mu_{t}(\d y) \geq -\frac{C_{\gamma}}{2}\int_{\R^{2}}|x|^{\gamma}\,|y|^{\gamma}\mu_{t}(\d x)\mu_{t}(\d y)
\end{equation*}
from which it follows that
\begin{equation*}
\dfrac{\d}{\d t}M_{\gamma}(\mu_{t}) \geq -\frac{C_{\gamma}}{2} M_{\gamma}(\mu_{t})^{2} \qquad \forall\, t \geq 0\,.
\end{equation*}
This inequality, after integration, yields \eqref{eq:Mg} with $\mathcal{K}_{\gamma}=\frac{C_{\gamma}}{2}$.   Inequality \eqref{eq:lowM2} follows from the fact that according to H\"{o}lder inequality 
\begin{equation*}
M_{\gamma}(\mu_{t}) \leq M_{2}(\mu_{t})^{\frac{\gamma}{2}}\,\|\mu_{t}\|_{0}^{\frac{2-\gamma}{2}}\,,
\end{equation*}
while $\|\mu_{t}\|_{0}=\|\mu_0\|_{0}=1$ for any $t \geq 0$. 
\end{proof}

The remaining case $\gamma > 1$ is more involved.  In this case, inequality \eqref{eq:mga} no longer holds and we need the following result.
\begin{lem}\label{l1}
Fix $\gamma >1$.  For any $p > 1$ and $\mu \in \P_{p+\gamma}^{0}(\R)$ it holds that
\begin{equation}\label{ulb}
\langle \Q(\mu\,,\,\mu)\,;\,|\cdot|^{p}\rangle \leq - \big(1-\beta_{p}(\a)\big)M_{p+\gamma}(\mu) +  \mathcal{R}_{p}(\mu),
\end{equation}
where $\beta_{2p}(\a)=\max\{a^{p-1},b^{p-1}\}$ and 
\begin{equation*}
\mathcal{R}_{p}(\mu)=2^{\gamma-1}\beta_{p}(\a)  \sum^{\left[\frac{p+1}{2}\right]}_{k=1}\Big(\begin{array}{c}p\\ k\end{array}\Big) \,\Big(M_{k+\gamma}(\mu)\,M_{p -k}(\mu) + M_{k}(\mu)\,M_{p-k +\gamma}(\mu)\Big)\,.
\end{equation*}
Moreover, for $k \in (0,1]$
\begin{equation}\label{lb}
-\langle  \Q(\mu\,,\,\mu)\,;\,|\cdot|^{k}\rangle \leq C_{\gamma,k}\,M_{\gamma}(\mu)\,M_{k}(\mu)\,. 
\end{equation} 
The constant $C_{\gamma,k}$ depends only on $\gamma$ and $k$.
\end{lem}
\begin{proof} Let us begin with inequality \eqref{ulb}. We first notice the following elementary inequality, valid for any $p >0$,
\begin{equation} \label{eq:ele}
|\a x + \b y|^{p} + |\a y +\b x|^{p} \leq \beta_{p}(\a)\left(x^{2}+y^{2}\right)^{\frac{p}{2}} \leq \beta_{p}(\a)\left(|x|+|y|\right)^{p} \qquad \forall (x,y) \in \R^{2}\,,
\end{equation}
where $k \mapsto \beta_{k}(\a)$ is decreasing with $\beta_{2}(\a)=1$ and $\lim_{k \to \infty}\beta_{k}(\a)=0$ (recall that $\max\{\a,\b\} < 1$).  We use then the following useful result given in \cite[Lemma 2]{BGP} for estimation of the binomial for fractional powers.  For any $p\geq1$, if $k_p $ denotes the integer part of $\frac{p+1}{2}$, the following inequality holds for any $u,v \in \R_+$,
\begin{equation}\label{base}
\big(u + v\big)^{p} - u^{p} - v^{p}\leq \sum^{k_p}_{k=1}\Big(\begin{array}{c}p\\ k\end{array}\Big)\big(u^{k}v^{p-k} + u^{p-k}v^{k}\big)\,,
\end{equation} 
where the binomial coefficients are defined as
\begin{equation*}
\Bigg(\begin{array}{c}p\\ k\end{array}\Bigg)=\Bigg\{
\begin{array}{cl}
\frac{p(p-1)\cdots(p-k+1)}{k!} &\text{for }\;k\geq1\\
1 & \text{for }\; k=0\,. 
\end{array}
\end{equation*}
Therefore, for any $p > 1$,  one gets
\begin{align*}
|\a x+\b y|^{p}+|\a y + \b x|^{p}-|x|^{p}-|y|^{p}
& \leq -\big( 1 - \beta_{p}(\a)  \big)\big( |x|^{p} + |y|^{p}\big) \\
& + \beta_{p}(\a)\sum^{k_p}_{k=1}\Big(\begin{array}{c}p\\ k\end{array}\Big)\big(|x|^{k}\,|y|^{p-k} + |x|^{p-k}\,|y|^{k} \big).
\end{align*}
Therefore, using  \eqref{eq:Qmu} we get
\begin{equation*}
\begin{split}
\langle \Q(\mu\,,\,\mu)\,;\,|\cdot|^{p}\rangle   &\leq -\tfrac{1}{2}\big( 1 - \beta_{p}(\a)  \big)\int_{\R^{2}}\left(|x|^{p}+|y|^{p}\right)\big|x-y\big|^{\gamma}\mu(\d x)\mu(\d y)\\
&+ \frac{1}{2}\beta_{p}(\a) \sum^{k_p}_{k=1}\Big(\begin{array}{c}p\\ k\end{array}\Big)\int_{\R^{2}}|x-y|^{\gamma}\,\big(|x|^{k}\,|y|^{p-k} + |x|^{p-k}\,|y|^{k} \big)\mu(\d x)\mu(\d y).
\end{split}
\end{equation*}
Furthermore, since $\gamma >1$ using Jensen's inequality and the fact that $\mu \in \P^{0}(\R)$, one obtains
\begin{equation*}
\int_{\R^{2}}\left(|x|^{p}+|y|^{p}\right)\big|x-y\big|^{\gamma}\mu(\d x)\mu(\d y) \geq
2\,M_{p+\gamma}(\mu).
\end{equation*}
Consequently,  
\begin{equation*}\begin{split}
\langle \Q(\mu\,,\,\mu)\,;\,|\cdot|^{p}\rangle   &\leq -\big( 1 - \beta_{p}(\a)  \big) M_{p+\gamma}(\mu) 
\\
&+ \frac{1}{2}\;\beta_{p}(\a) \sum^{k_p}_{k=1}\Big(\begin{array}{c}p\\ k\end{array}\Big)\int_{\R^{2}}|x-y|^{\gamma}\,\big(|x|^{k}\,|y|^{p-k } + |x|^{p-k }\,|y|^{k} \big)\mu (\d x)\mu (\d y).
\end{split}
\end{equation*}
For the  remainder term, one uses the inequality $|x-y|^{\gamma} \leq  2^{\gamma-1} \big( |x|^{\gamma} + |y|^{\gamma} \big)$ to get that
\begin{multline*}
 \int_{\R^{2}}|x-y|^{\gamma}\,\big(|x|^{k}\,|y|^{p-k} + |x|^{ p-k }\,|y|^{ k} \big)\mu (\d x)\mu (\d y) \\
\leq 2^{\gamma} \big(M_{ k+\gamma}(\mu )\,M_{p-k}(\mu)+
  M_{k}(\mu )\,M_{ p-k +\gamma}(\mu ) \big)\,,
\end{multline*}
which proves \eqref{ulb}. The proof of \eqref{lb} relies on inequality \eqref{eq:mga}. Indeed, for any $\mu$ and any $k \in (0,1)$ one has
\begin{equation*}
\langle \Q(\mu\,,\,\mu)\,;\,|\cdot|^{k}\rangle =\dfrac{1}{2}\int_{\R^{2}}\Big(\left|\a x + \b y \right|^{k} + \left|\a y  + \b x \right|^{k}- |x|^{k}- |y|^{k}\Big)|x-y|^{\gamma}\,\mu(\d x)\mu(\d y)\,.
\end{equation*}
Now, by virtue of \eqref{eq:mga} with $k \in (0,1)$ instead of $\gamma$
\begin{equation*}
\Big(\left|\a x + \b y \right|^{k} + \left|\a y  + \b x \right|^{k}- |x|^{k}- |y|^{k}\Big)|x-y|^{k} \geq -C_{k}|x|^{k}\,|y|^{k} \qquad \forall\, x,y \in \R\,,
\end{equation*}
for some positive constant $C_{k} >0$ explicit and depending only on $\a$ and $k$. Therefore,
\begin{equation*}
-\langle \Q(\mu\,,\,\mu)\,;\,|\cdot|^{k}\rangle \leq \dfrac{C_{k}}{2}\int_{\R^{2}}|x|^{k}\,|y|^{k}\,|x-y|^{\gamma-k}\mu(\d x)\mu(\d y)\,.
\end{equation*}
Using that $|x-y|^{\gamma-k}\leq { \max\{1,2^{\gamma-k-1}\}}\,(|x|^{\gamma-k}+|y|^{\gamma-k})$ we get
\begin{equation*}
\begin{split}
-\langle \Q(\mu\,,\,\mu)\,;\,|\cdot|^{k}\rangle &\leq { \frac{1}{2}\max\{1,2^{\gamma-k-1}\}}\, C_{k}\int_{\R^{2}}|x|^{k}\,|y|^{k}\left(|x|^{\gamma-k}+|y|^{\gamma-k}\right)\mu(\d x)\mu(\d y)\\
&={\max\{1,2^{\gamma-k-1}\}}\, C_{k}M_{\gamma}(\mu)M_{k}(\mu)\end{split}\end{equation*}
which gives the proof with $C_{\gamma,k}= {\max\{1,2^{\gamma-k-1}\}}\,C_{k}.$\end{proof}
As a consequence, the following proposition.
\begin{prop}\label{l2}
For any $\gamma >1$ and $s\in(0,1]$ it follows that
\begin{equation}\label{cineq}
M_{s+\gamma}(\mu_{t}) \leq A_{s}\,M_{s}(\mu_{t})^{1+ \frac{\gamma}{s}}\quad\quad t\geq0\,,
\end{equation}
with $A_{s}$ any constant such that
\begin{equation}\label{cte}
A_{s}\geq\max\left\{\tilde C_{\gamma,s},\,\frac{M_{s+\gamma}(\mu_0)}{M^{1+\frac{\gamma}{s}}_{s}(\mu_0)}\right\}\,,
\end{equation}
and $\tilde C_{\gamma,s}$ is a constant depending only on $\gamma$ and $s$.
\end{prop}
\begin{proof}
Define $X(t) := M_{s+\gamma}(\mu_t) - A\,M_{s}(\mu_{t})^{1+ \frac{\gamma}{s}}$ where the constant $A$ will be conveniently chosen later on. Since $\gamma+s >1$ and  $s\in(0,1]$ we can use Lemma \ref{l1} to conclude that
\begin{align}
\frac{\d}{\d t}X(t) &= \frac{\d}{\d t}M_{s+\gamma}(\mu_{t}) - A\,\left(1+\frac{\gamma}{s}\right) M_{s}(\mu_{t})^{\frac{\gamma}{s}} \frac{\d}{\d t}M_{s}(\mu_{t})\nonumber\\
&\hspace{-.5cm}\leq -\left(1-\beta_{s+\gamma}(\a) \right)M_{s+2\gamma}(\mu_{t})+ \mathcal{R}_{s+\gamma}(\mu_{t}) + A\,C_{\gamma,s} \left(1+\frac{\gamma}{s}\right)M_{s}(\mu_{t})^{1+\frac{\gamma}{s}}\,M_{\gamma}(\mu_{t}) \,.\label{ar0}
\end{align}
Let us observe that for such a choice of $s$ and $\gamma$, one has 
$ \left[\frac{s+\gamma+1}{2}\right]<s+\gamma$ and for $1\leq k\leq \left[\frac{s+\gamma+1}{2}\right]$, a simple use of H\"older and Young's inequalities  leads to, for any $\epsilon > 0$, 
\begin{equation*}
M_{k+\gamma}(\mu_{t})\,M_{s+\gamma - k}(\mu_{t}) \leq M_{s+ 2\gamma}(\mu_{t})^{\frac{k + \gamma}{s+ 2\gamma}}\,M_{s+\gamma}(\mu_{t})^{\frac{s+\gamma-k}{s+\gamma}} \leq \epsilon\, M_{s+2\gamma}(\mu_{t}) + K_{s,\gamma,\epsilon}\, M_{s+\gamma}(\mu_{t})^{\frac{s+2\gamma}{s+\gamma}}
\end{equation*}
with $K_{s,\gamma,\epsilon}>0$ a constant depending only on $s$, $\gamma$ and $\epsilon$.  Similar interpolation holds  for the terms of the form $M_{k}(\mu_{t})\,M_{s+2\gamma - k}(\mu_{t})$ appearing in $\mathcal{R}_{s+\gamma}(\mu_{t})$.  As a consequence, for any $\epsilon > 0$, we have
\begin{equation}\label{ar1}
\mathcal{R}_{s+\gamma}(\mu_{t}) \leq \epsilon\,K_{s,\gamma}\,M_{s+2\gamma}(\mu_t) + K_{s,\gamma}\,K_{s,\gamma,\epsilon}\, M_{s+\gamma}(\mu_{t})^{\frac{s+2\gamma}{s+\gamma}}\,,
\end{equation}
where $K_{s,\gamma}>0$ is a constant depending only on $s$ and $\gamma$.   Choose $\epsilon=\epsilon(\gamma)>0$ such that
\begin{equation*} 
2\,K_{s,\gamma}\,\epsilon \leq \ 1-\beta_{s+\gamma}(\a)
\end{equation*}
we obtain from estimates \eqref{ar0} and \eqref{ar1} that there are constants $K_{1}=K_{1}(s,\gamma)$ and $K_{2}=K_{2}(s,\gamma)$ such that
\begin{align}
\frac{\d}{\d t}X(t) &\leq -\frac{1}{2}\left(1-\beta_{s+\gamma}(\a) \right)M_{s+2\gamma}(\mu_{t})\nonumber \\
&+ A\,K_{1}\,M_{s}(\mu_{t})^{1+\frac{\gamma}{s}+\frac{s}{\gamma}}\,M_{s+\gamma}(\mu_{t})^{1-\frac{s}{\gamma}}+
K_2\, M_{s+\gamma}(\mu_{t})^{\frac{s+2\gamma}{s+\gamma}}\,,\label{ar0.5}
\end{align}
where we also used the interpolation
\begin{equation*}
M_{\gamma}(\mu_{t})\leq M_{s}(\mu_{t})^{\frac{s}{\gamma}}\,M_{s+\gamma}(\mu_{t})^{1-\frac{s}{\gamma}}\,.
\end{equation*}
As a final step, notice that 
\begin{equation}\label{ar3}
\frac{M_{s+\gamma}(\mu_{t})^{2}}{M_{s}(\mu_{t})} \leq M_{s+2\gamma}(\mu_{t})\,.
\end{equation}
Therefore, including  \eqref{ar3} in \eqref{ar0.5} one finally concludes that
\begin{align}
\frac{\d}{\d t} X(t) & \leq -\frac{1}{2}\left(1-\beta_{s+\gamma}(\a) \right)M_{s+\gamma}(\mu_{t})^{2}\, M_{s}(\mu_{t}) \nonumber \\
&\hspace{-.5cm} + A\,K_{1}\,M_{s}(\mu_{t})^{1+\frac{\gamma}{s}+\frac{s}{\gamma}}\,M_{s+\gamma}(\mu_{t})^{1-\frac{s}{\gamma}}+
K_2\, M_{s+\gamma}(\mu_{t})^{\frac{s+2\gamma}{s+\gamma}}\,.\label{ar4}
\end{align}
Now, choosing $A > 0$ such that $X(0) < 0$, if there exists $t_0>0$ for which $X(t_0)=0$, then estimate \eqref{ar4} implies that
\begin{equation}
\frac{\d}{\d t} X(t_0) \leq \Big(-\frac{1}{2}\left(1-\beta_{s+\gamma}(\a) \right)A^{2} +K_{1}\,A^{2-\frac{s}{\gamma}} + K_{2}\,A^{1+\frac{\gamma}{s+\gamma}}\Big) M_{s}(\mu_{t_{0}})^{1+ \frac{2\gamma}{s}}\,.\label{ar6}
\end{equation}
Then, choosing $A=A(\gamma,s)$ sufficiently large such that the term in parenthesis in \eqref{ar6} is negative we conclude that  $X'(t_0) \leq 0$. This shows that, for such a choice of $A$, $X(t) \leq 0$ for any $t \geq 0$.
\end{proof}
Using Proposition \ref{l2} the desired lower bound is obtained.
\begin{prop}\label{p2}
For any $\gamma >1$ and $s\in(0,1]$ one has
\begin{equation}\label{lower}
M_{s}(\mu_{t}) \geq \frac{M_{s}(\mu_0)}{\big(1+C_{\gamma,s}\,A_{s}^{1-\frac{s}{\gamma}} \frac{\gamma}{s}\, M_{s}(\mu_{0})^{ \frac{\gamma}{s} }\, t \big)^{ \frac{s}{\gamma} } }\,,
\end{equation}
where $C_{\gamma,s}$ depends only on $\gamma$ and $s$ and $A_s$ is given by \eqref{cte}. Moreover,  it holds 
\begin{equation}\label{eq:Mk1}
M_{k}(\mu_{t})\geq \frac{M_{s}(\mu_{0})^{\frac{k}{s}}}{\big(1+C_{\gamma,s}\,A_{s}^{1-\frac{s}{\gamma}} \frac{\gamma}{s}\, M_{s}(\mu_{0})^{ \frac{\gamma}{s} }\, t \big)^{ \frac{k}{\gamma} } }\,,\quad \forall\, k > 1,\;s \in (0,1]\;\; \text{and}\;\; t \geq 0\,.
\end{equation}
\end{prop}
\begin{proof} Using  Lemma \ref{l1} and inequality \eqref{cineq}, since $s \in (0,1)$ one gets
\begin{align*}
-\frac{\d }{\d t}M_{s}(\mu_{t}) &= -\langle \Q(\mu_{t},\mu_{t})\,;\, |\cdot|^{s}\rangle\leq C_{\gamma,s}\,M_{s}(\mu_{t})\,M_{\gamma}(\mu_{t}) \\
&\leq C_{\gamma,s}\,M_{s}(\mu_{t})^{1+\frac{s}{\gamma}}\,M_{s+\gamma}(\mu_{t})^{1-\frac{s}{\gamma}} \leq C_{\gamma,s}\,A_{s}^{1-\frac{s}{\gamma}}\,M_{s}(\mu_{t})^{1+\frac{\gamma}{s}}\,,
\end{align*}
where, for the second estimate, we used the inequality 
$$ M_{\gamma}(\mu_{t}) \leq M_{s}^{\frac{s}{\gamma}}(\mu_{t}) \,M_{s+\gamma}^{1-\frac{s}{\gamma}}(\mu_{t}).$$
 Integration of this differential inequality leads to
\begin{equation*}
M_{s}(\mu_{t}) \geq \frac{M_{s}(\mu_0)}{\big(1+C_{\gamma,s}\,A_{s}^{1-\frac{s}{\gamma}} \frac{\gamma}{s}\, M_{s}(\mu_{0})^{ \frac{\gamma}{s} }\, t \big)^{ \frac{s}{\gamma} } }\,.
\end{equation*}
The second part of the result follows from 
$$
M_{s}(\mu_{t}) \leq M_{k}^{\frac{s}{k}}(\mu_{t})\,M_{0}(\mu_{t})^{1-\frac{s}{k}}= M_{k}^{\frac{s}{k}}(\mu_{t}),$$
valid for any $k >1$.
\end{proof}
We can now completely characterize the  decay of any moments of the weak measure solution $(\mu_t)_{t\geq 0}$ associated to $\mu_{0}.$ Namely,
\begin{theo} \label{theo:decmom}
Let $k \geq 2$ and $\mu_{0} \in \P_{k+\gamma}^{0}(\R)$ be given. Denote by $(\mu_{t})_{t\geq 0}$ a measure weak solution to \eqref{eq:1} associated to the initial datum $\mu_{0}$. Then,\\

\noindent
(1) When $\gamma \in (0,1]$, there exists some universal constant $\mathcal{K}_{\gamma} > 0$ (not depending on $\mu_{0}$) such that for any  $p\geq \gamma$
\begin{subequations}
\begin{equation}\label{eq:decmom}
\dfrac{M_{\gamma}\left(\mu_{0}\right)^{\frac{p}{\gamma}}}{\Big(1+\mathcal{K}_{\gamma}\,M_{\gamma}(\mu_{0})\,t \Big)^{\frac{p}{\gamma}}} \leq M_{p}\left(\mu_{t}\right) \qquad \forall\, t \geq 0\,,
 \end{equation}
and for any $p\in[2,k]$
\begin{equation}\label{eq:decmomb}
 M_{p}\left(\mu_{t}\right)\leq \dfrac{M_{p}\left(\mu_{0}\right)}{\Big(1+\frac{\gamma}{2p}\left(1-\a^{p}-\b^{p}\right)M_{p}\left(\mu_{0}\right)^{\frac{\gamma}{p}}\,t \Big)^{\frac{p}{\gamma}}} \qquad \forall\, t \geq 0\,.
 \end{equation}
\end{subequations}
In particular, if $\mu_{0} \in \P_{\exp,\gamma}(\R)$  then \eqref{eq:decmomb} holds true for any  $p\geq2$.\\

\noindent
(2) When $\gamma >1$, for any $p> 1$,
\begin{subequations}
\begin{equation}\label{eq:decmom1}
\dfrac{M_1\left(\mu_{0}\right)^{p}}{\Big(1+\gamma\,C_{\gamma,1}\,A_{1}^{1-\frac{1}{\gamma}}\,M_{1}(\mu_{0})^{\gamma}\,t \Big)^{\frac{p}{\gamma}}} \leq M_{p}\left(\mu_{t}\right)\qquad \forall\, t \geq 0\,,
\end{equation}
and for any $p\in[2,k]$
\begin{equation}\label{eq:decmom1b}
 M_{p}\left(\mu_{t}\right)\leq \dfrac{M_{p}\left(\mu_{0}\right)}{\Big(1+\frac{\gamma}{2p}\left(1-\a^{p}-\b^{p}\right)M_{p}\left(\mu_{0}\right)^{\frac{\gamma}{p}}\,t \Big)^{\frac{p}{\gamma}}} \qquad \forall\, t \geq 0\,,
\end{equation}
\end{subequations}
where $C_{\gamma,1}$ and $A_{1}$ are defined in Proposition \ref{p2}.
\end{theo}
\begin{proof}
The upper bounds in \eqref{eq:decmomb} and \eqref{eq:decmom1b} have been established in Proposition \ref{prop:decrease}. For the lower bound, whenever $\gamma \in (0,1)$, one simply uses Jensen's inequality to get 
\begin{equation*}
M_{p}(\mu_{t}) \geq M_{\gamma}(\mu_{t})^{\frac{p}{\gamma}} \qquad \forall\, t \geq 0,\;\;p \geq \gamma\,,
\end{equation*}
and, then conclude thanks to equation \eqref{eq:Mg}. For $\gamma >1$, the lower bound is just \eqref{eq:Mk1} with $s=1$.
\end{proof}
\begin{cor}
Fix $\gamma >0$ and let $\mu_{0} \in \P^{0}(\R) \cap \P_{\exp,\gamma}(\R)$ be an initial datum. Then, for any $p \geq 1$, the unique measure weak solution $(\mu_{t})_{t \geq 0}$ is converging as $t \to \infty$ towards $\delta_{0}$ in the weak-$\star$ topology of $\P_{p}(\R)$ with the explicit rate
\begin{equation*} 
W_{p}(\mu_{t},\delta_{0}) \propto \big(1+C\,t\big)^{-\frac{1}{\gamma}} \quad \text{as}\;\; t \to \infty\,,
\end{equation*}
for some positive constant $C$ depending on $\mu_{0}$.
\end{cor}
\begin{proof} The result is a direct consequence of Theorem \ref{theo:decmom} since 
\begin{equation*}
W_{p}(\mu_{t},\delta_{0})^{p}=\int_{\R}|x|^{p}\mu_{t}(\d x)=M_{p}(\mu_{t}) \qquad t \geq 0\,,
\end{equation*}
and $W_{p}$ metrizes $\P_{p}(\R)$ (see \cite[Theorem 6.9]{vil2}).
\end{proof}
\subsection{Consequences on the rescaled problem}\label{sec:resca}
Given $\mu_{0} \in \P_{\gmm}^{0}(\R)$ satisfying \eqref{eq:exp} and let $(\mu_{t})_{t\geq 0}$ denotes the unique weak measure solution to \eqref{eq:1} associated to $\mu_{0}.$ Recall from Section \ref{sec:intro} the definitions of the rescaling functions
$$V(t):=(1+c\,\gamma\, t)^{\frac{1}{\gamma}} \quad \text{ and } \quad s(t):=\frac{1}{c\,\gamma}\log(1+c\,\gamma\, t), \qquad (c > 0).$$
For simplicity, in all the sequel, we shall assume 
$c=1.$ The inverse mappings are defined as
\begin{equation*}
\V(s)=\exp(s)\,\quad t(s)=\dfrac{\exp(\gamma\, s)-1}{\gamma}\,,
\end{equation*}
i.e.
\begin{equation*}
s(t)=s \Longleftrightarrow t(s)=t \quad \text{ and }\quad  V\big(t(s)\big)=\V(s) \quad\forall\, t,s \geq 0\,.
\end{equation*}
In this way we may define, for any $s \geq 0$, the measure $\nub_{s}$ as the image of $ \mu_{t(s)}$ under the transformation  $x \mapsto \V(s)x$
\begin{equation*}
\nub_{s}(\d \xi)= \big(\V(s) \# \mu_{t(s)}\big)(\d \xi) \qquad \forall\, s \geq 0\,,
\end{equation*}
where $\#$ stands for the push-forward operation on measures, 
\begin{equation}\label{eq:phinus}
\int_{\R}\phi(\xi)\nub_{s}(\d \xi)= \int_{\R}\phi(\V(s)x)\mu_{t(s)}(\d x) \qquad \forall\, \phi \in \C_{b}(\R),\,s \geq 0\,.
\end{equation}
Notice that, whenever $\mu_{t}$ is absolutely continuous with respect to the Lebesgue measure over $\R$ with $\mu_{t}(\d x)=f(t,x)\d x$, the measure $\nub_{s}$ is also absolutely continuous with respect to the Lebesgue measure over $\R$ with $\nub_{s}(\d \xi)=g(s,\xi)\d \xi$ where
\begin{equation*}
g(s,\xi)=\V(s)^{-1}\,f\big(t(s),\V(s)^{-1}\xi\big) \qquad \forall\, \xi \in \R,\, s \geq 0\,,
\end{equation*}
which is nothing but \eqref{eq:change}.  Such a definition of $\nub_{s}$ allows to define the semi-flow $(\mathcal{F}_{s})_{s \geq 0}$ which given any initial datum $\mu_{0} \in \P^{0}(\R)$ satisfying \eqref{eq:exp} associates
\begin{equation*}
\mathcal{F}_{s}(\mu_0)= \nub_{s}= \V(s) \# \mathcal{S}_{t(s)}(\mu_{0})  \qquad \forall\, s \geq 0\,,
\end{equation*}
where $(\mathcal{S}_{t})_{t \geq 0}$ is the semi-flow associated to equation \eqref{eq:1}.   The decay of the moments given by Theorem \ref{theo:decmom} readily translates into the following result.
\begin{theo}\label{theo:decmoms}
Let $\gamma >0$ and let $\mu_{0} \in \P^{0}(\R)\cap\P_{\exp,\gamma}(\R)$ be a given initial datum. Denote by $\mathcal{F}_{s}(\mu_0)= \nub_{s}$ for any $s \geq 0$.   Then,\\

\noindent
(1) When $\gamma \in (0,1]$, there exists some universal constant $\mathcal{K}_{\gamma} > 0$ (not depending on $\mu_{0}$) such that for any $p\geq \gamma$
\begin{subequations}
\begin{equation}\label{eq:decmoms}
\min\bigg\{M_{\gamma}\left(\mu_{0}\right)^{\frac{p}{\gamma}}\,;\,\bigg(\dfrac{\gamma}{\mathcal{K}_{\gamma}}\bigg)^{\frac{p}{\gamma}}\bigg\}  \leq M_{p}\left(\nub_{s}\right)  \qquad \forall\, s \geq 0\,,
\end{equation}
and for any $p\geq2$,
\begin{equation}\label{eq:decmomsb}
 M_{p}\left(\nub_{s}\right)  
\leq \max\bigg\{M_{p}(\mu_{0})\,;\,\bigg(\dfrac{2p}{1-\a^{p}-\b^{p}}\bigg)^{\frac{p}{\gamma}}\bigg\} \qquad \forall\, s \geq 0\,.
\end{equation}
\end{subequations}
(2) When $\gamma >1$,  for any $p>1$
\begin{subequations}
\begin{equation}\label{eq:decmom2}
\min\bigg\{M_{1}\left(\mu_{0}\right)^{p}\,;\,\left(\dfrac{1}{C_{\gamma,1}\,A_{1}^{1-\frac{1}{\gamma}}}\right)^{\frac{p}{\gamma}}\bigg\}  \leq M_{p}\left(\nub_{s}\right)  
 \qquad \forall\, s \geq 0\,,
\end{equation}
and for any $p\geq 2$
\begin{equation}\label{eq:decmom2b}
 M_{p}\left(\nub_{s}\right)  
\leq \max\bigg\{M_{p}(\mu_{0}) \,;\,\bigg(\dfrac{2p}{1-\a^{p}-\b^{p}}\bigg)^{\frac{p}{\gamma}}\bigg\} \qquad \forall\, s \geq 0\,,
\end{equation}
\end{subequations}
where $C_{\gamma,1}$ and $A_{1}$ are defined in the above Proposition \ref{p2}.
\end{theo}
\begin{proof} The proof follows simply from the fact that
\begin{equation*}
M_{p}(\nub_{s})=\V(s)^{p}\,M_{p}(\mu_{t(s)}) \qquad \forall\, s \geq 0\,,
\end{equation*}
where $(\mu_{t})_{t}$ is the weak measure solution to \eqref{eq:1} associated to $\mu_{0}$.  Then, according to Theorem \ref{theo:decmom} and using that $\exp(p\,s)=\big(1+\gamma\,t(s)\big)^{p/\gamma}$, we see that for $\gamma \in (0,1]$ it holds
\begin{equation*}
 \left(\dfrac{M_{\gamma}(\mu_{0})(1+\gamma t(s))}{1+\mathcal{K}_{\gamma}\,M_{\gamma}(\mu_{0})\,t(s)}\right)^{\frac{p}{\gamma}}\leq M_{p}\left(\nub_{s}\right)  \leq 
M_{p}(\mu_{0}) \bigg(\dfrac{1+\gamma t(s)}{1+\frac{\gamma}{2p}\left(1-\a^{p}-\b^{p}\right)M_{p}\left(\mu_{0}\right)^{\frac{\gamma}{p}}\,t(s)}\bigg)^{\frac{p}{\gamma}}\,,
\end{equation*}
where the lower bound is valid for any $p\geq \gamma$ while the upper bound is valid for any $p\geq 2$.
Since $\min(1, \frac{A}{B}) \leq \frac{1+At}{1+Bt} \leq \max(1,\frac{A}{B})$ for any $A,B, t \geq 0$ we get the conclusion. We proceed in the same way for $\gamma >1.$
\end{proof}
An important consequence of the above decay is the following proposition.
\begin{prop}\label{prop:expo}
Let $\mu_{0} \in \cap_{k\geq0}\P_{k}^{0}(\R)$ be a given initial condition.  Assume that
\begin{equation}\label{eq:allmom}
M_{p}(\mu_{0}) \leq  {\mathbf{M}_{p}}  \qquad \forall\, p \geq \max(\gamma,2)\,,
\end{equation}
where
\begin{equation*}
{\mathbf{M}_{p}}:=\left(\dfrac{2p}{1-\a^{p}-\b^{p}}\right)^{\frac{p}{\gamma}} \qquad \forall\, p \geq \max(\gamma,2)\,.
\end{equation*}
Then, $\mu_{0} \in \P_{\exp,\gamma}(\R)$ and there exists an explicit  $\alpha > 0$ and $C=C(\alpha) > 0$ such that
\begin{equation*}
\sup_{s \geq 0}\int_{\R}\exp(\alpha |\xi|^{\gamma}) \,\nub_{s}(\d \xi) \leq C\,,
\end{equation*}
where $\nub_{s}=\mathcal{F}_{s}(\mu_{0})$ for any $s \geq 0$.
\end{prop}
\begin{proof}
Let us first prove that $\mu_{0} \in \P_{\exp,\gamma}(\R)$. Notice that for any $z> 0$
\begin{equation*}
\int_{\R}\exp(z \,|\xi|^{\gamma})\mu_{0}(\d\xi)= \sum_{p=0}^{\infty}M_{\gamma\,p}(\mu_{0}) \dfrac{z^{p}}{p!}\,.
\end{equation*}
Let us denote by $p_0$ the integer such that $\gamma\, p_0\geq \max(\gamma,2)$ and $\gamma\,(p_0-1)<\max(\gamma,2)$. Using Stirling formula together with the fact that $\lim_{p \to \infty}(1-\a^{\gamma\,p}-\b^{\gamma\,p})=1$, one can check that there exists some explicit $\alpha > 0$ such that the series
\begin{equation*}
\sum_{p=p_0}^{\infty}  {\mathbf{M}_{\gamma\,p}} \dfrac{z^{p}}{p!}
\end{equation*}
converges for any $0\leq z \leq \alpha$ which gives the result.  Now, we may define $\nub_{s}=\mathcal{F}_{s}(\mu_{0})$ for any $s \geq 0$. As previously, for any $z> 0$
\begin{equation*}
\int_{\R}\exp(z \,|\xi|^{\gamma})\nub_{s}(\d\xi)= \sum_{p=0}^{\infty}M_{\gamma\,p}(\nub_{s}) \dfrac{z^{p}}{p!}\,,
\end{equation*}
  and we deduce from Theorem \ref{theo:decmoms} that
\begin{equation*}
M_{p}(\nub_{s}) \leq  {\mathbf{M}_{p}}  \qquad \forall\,  p \geq \max(\gamma,2),\, s \geq 0\,.
\end{equation*}
The conclusion follows.
\end{proof}
\begin{nb} We do not need to derive the equation satisfied by $\nub_{s}$ since we are interested only in the fixed point of the semiflow. However, using the fact that $\mathcal{S}_{t}(\mu_{0})$ actually provides a \emph{strong solution} to \eqref{eq:1}, using the chain rule it follows that
\begin{equation}\label{eq:2}
\int_{\R}\phi(\xi)\nub_{s}(\d \xi)=\int_{\R}\phi(\xi)\mu_{0}(\d \xi) + \int_{0}^{s}\d\tau 
\int_{\R} \xi \phi'(\xi) \nub_{\tau}(\d \xi) +\int_{0}^{s}\la \Q\left(\nub_{\tau}\,,\,\nub_{\tau}\right)\,;\,\phi \ra \d\tau \quad \forall\, s \geq 0\,,
\end{equation}
for any $\phi \in \C^{1}_{b}(\R)$ and where $\phi'$ stands for the derivative of $\phi$.
\end{nb} 
The link between solution to \eqref{meas_form} and the semiflow $(\mathcal{F}_{s})_{s\geq0}$ is established by the following lemma.
\begin{lem}
Any fixed point $\mub \in \P^{0}(\R)\cap \P_{\exp,\gamma}(\R)$ of the semi-flow $(\mathcal{F}_{s})_{s\geq 0}$ is a solution to \eqref{meas_form}.
\end{lem}
\begin{proof} Let $\mub$ be a fixed point of the semi-flow $(\mathcal{F}_{s})_{s}$, that is, $\mathcal{F}_{s}(\mub)=\mub$ for any $s \geq 0$.  Then, according to \eqref{eq:phinus}, for any $\phi \in \C_b(\R)$
\begin{equation*}
\int_{\R}\phi(\xi)\mub(\d \xi)=\int_{\R}\phi(\V(s)x)\mu_{t(s)}(\d x) \qquad \forall\, s \geq 0\,,
\end{equation*}
where $\mu_{t(s)}=\mathcal{S}_{t(s)}(\mub)$.  In particular, choosing $s=s(t)$
\begin{equation*}
\int_{\R}\phi(\xi)\mub(\d \xi)=\int_{\R}\phi(V(t)x)\mu_{t}(\d x) \qquad \forall\, t \geq 0\,.
\end{equation*}
Applying the above to $\phi\big(V(t)^{-1}\cdot\big)$ instead of $\phi$, one obtains
\begin{equation*}
\int_{\R}\phi\big(V(t)^{-1}\xi\big)\mub(\d \xi)=\int_{\R}\phi(x)\mu_{t}(\d x) \qquad \forall\, t \geq 0\,.
\end{equation*}
Computing the derivative with respect to $t$ and assuming $\phi \in \C^{1}_b(\R)$, we get
\begin{equation*}
\dfrac{\d}{\d t}\big(V(t)^{-1}\big)\int_{\R}\xi \phi'\big(V(t)^{-1}\xi\big)\mub(\d \xi)= \la \Q(\mu_{t},\mu_{t})\,;\,\phi\ra \qquad \forall\, t \geq 0\,.
\end{equation*}
Using the definition of $V(t)$ it finally follows that
\begin{equation*}
-\big(1+\gamma t\big)^{-\frac{1}{\gamma}-1}\int_{\R}\xi \phi'\big(V(t)^{-1}\xi\big)\mub(\d \xi)=\la \Q(\mu_{t},\mu_{t})\,;\,\phi \ra \qquad \forall\, t \geq 0\,.
\end{equation*}
Taking in particular $t=0$ it follows that $\mub$ satisfies \eqref{meas_form}.
\end{proof}
\section{Existence of a steady solution to the rescaled problem}\label{sec=exi}
\subsection{Steady measure solutions are $L^{1}$ steady state}
In this section we prove Theorem \ref{theo:meas-L1}, that is, we prove that steady measure solutions to \eqref{steady} are in fact $L^{1}$ functions provided no mass concentration happens at the origin.  The argument is based on the next two propositions.
\begin{prop}\label{prop:Hmu}
Let $\mub \in \P_{\max(\gamma,2)}^0(\R)$ be a steady solution to \eqref{steady}. Then, there exists $H \in  L^{1}(\R)$ such that
\begin{equation*}
\xi\,\mub(\d \xi)=H(\xi) \d\xi\,.
\end{equation*}
\end{prop}
\begin{proof}
Introduce the distribution $\boldsymbol \beta(\xi):=\xi \mub(\d\xi)$ which, of course, is defined by the identity
\begin{equation*}
\int_{\R}\boldsymbol \beta(\xi)\psi(\xi)\d\xi=\int_{\R} \xi \psi(\xi)\mub(\d\xi) \quad \text{ for any } \psi \in \mathcal{C}^{\infty}_{c}(\R)\,.
\end{equation*}
One sees from \eqref{meas_form} that $\boldsymbol \beta$ satisfies
\begin{equation}\label{eq:diff}
\dfrac{\d}{\d \xi}\boldsymbol \beta(\xi)=\Q(\mub,\mub)
\end{equation}
in the sense of distributions.  Since $\mub \in \P_{\max(\gamma,2)}^0(\R)$,  it follows that  $\Q^{\pm}(\mub,\mub)$ belongs to $\M^{+}(\R)$.  Therefore, as a solution to \eqref{eq:diff}, the measure $\boldsymbol \beta$ is a distribution whose derivative belongs to $\M(\R)$.  It follows from \cite[Theorem 6.77]{DemX2} that $\boldsymbol \beta\in BV_{\mathrm{loc}}(\R)$ where $BV(\R)$ denotes the space of functions with bounded variations. This implies that the measure $\boldsymbol\beta$ is absolutely continuous. In particular, there exists $H \in L^1(\R)$ such that $\boldsymbol \beta(\d\xi)=H(\xi)\d\xi$. This proves the result.
\end{proof}
\begin{prop}\label{prop:decomp}
Let $\mub \in \P_{\max(\gamma,2)}^0(\R)$ be a steady solution to \eqref{steady}. Then, there exist   $\kappa_0 \geq 0$ and $G \in L^1_{\max(\gamma,2)}(\R)$ nonnegative such that
\begin{equation*}
\mub(\d\xi) = G(\xi)\d\xi + \kappa_{0}\, \delta_0(\d\xi)
\end{equation*}
where $\delta_{0}$ is the Dirac mass in $0$.
\end{prop}
\begin{proof} 
Let us denote by $\mathcal{B}(\R)$ the set of Borel subsets of $\R$. According to Lebesgue decomposition Theorem \cite[Theorem 8.1.3]{stroock} there exists $G \in L^1(\R)$ nonnegative and a measure $\mub_s$ such that
\begin{equation*}
\mub(\d\xi) = G(\xi)\d\xi + \mub_s(\d \xi)
\end{equation*}
where the measure $\mub_s$ is singular to the Lebesgue measure over $\R$.  More specifically, there is $\Gamma\in \mathcal{B}(\R)$ with zero Lebesgue measure such that $\mub_s(\R\backslash \Gamma) = 0$. The proof of the Lemma consists then in proving that $\mub_s$ is supported in $\{0\}$, i.e. $\Gamma = \{0\}$. This comes directly from Proposition \ref{prop:Hmu}.  Indeed, by uniqueness of the Lebesgue decomposition, one has
\begin{equation*}
\xi\mub(\d\xi) =  \xi G(\xi)\d\xi +  \xi \mub_s(\d\xi) = \,H(\xi)\d\xi\,,
\end{equation*}
with {$H \in L^1(\R)$}, so that
\begin{equation*}
\xi \mub_s(\d\xi) = 0\,.
\end{equation*}
This implies that $\mub_s(\R \backslash (-\delta; \delta)) = 0$ for any $\delta > 0$ and therefore that $\mub_s$ is supported in $\{0\}$.  Notice that, since $\mub \in \P_{\max(\gamma,2)}^0(\R)$ one has $G \in L^{1}_{\max(\gamma,2)}(\R)$.
\end{proof}
\begin{proof}[Proof of Theorem \ref{theo:meas-L1}]
With the notations of Proposition \ref{prop:decomp}, our aim is to show that $\kappa_0 = 0$. 
Plugging the decomposition obtained in Proposition \ref{prop:decomp} in the weak 
formulation \eqref{meas_form}, we get 
\begin{equation*}
- \int_{\R} \xi \phi'(\xi)G(\xi)\d\xi =\int_{\R} \Q(G, G)(\xi)\phi(\xi)\d\xi
+ \kappa_{0}\int_{\R} |\xi|^\gamma \Big(\phi\left(\a\, \xi\right) +\phi\left(\b\, \xi\right) -\phi(\xi)- \phi(0)\Big) G(\xi) \d\xi\,,
\end{equation*}
where we used that $\Q(\delta_0, \delta_0) = 0 = \ds\int_{\R} \xi \phi'(\xi)
\delta_0(\d\xi)$.  Recall the hypothesis
\begin{equation}\label{contra}
\mathbf{m}_\gamma=\ds\int_\R|\xi|^\gamma G(\xi)\d\xi>0\,,
\end{equation}
from which the above can be reformulated as  
\begin{align*}
\kappa_{0}\,\mathbf{m}_\gamma \phi(0) 
&= \int_{\R} \Q(G, G)(\xi) \phi(\xi)\d\xi \\
&+  \int_{\R} \xi \phi'(\xi)G(\xi)\d\xi
+  \kappa_{0} \int_{\R} |\xi|^\gamma \Big(\phi\left(\a\, \xi\right) +\phi\left(\b\, \xi\right) -\phi(\xi)\Big) G(\xi) \d\xi
\end{align*}
for any  $\phi\in\mathcal{C}^1_b (\R)$. Notice that the above identity can be rewritten as
\begin{equation}\label{kapphi}
\kappa_{0}\, \phi(0) \mathbf{m}_\gamma 
= \int_{\R} \big(A(\xi) \phi(\xi) + B(\xi) \phi'(\xi)\big)\d\xi
\end{equation}
for some $L^1$-functions
\begin{equation*}
A (\xi) = \Q(G, G)(\xi) + \frac{\kappa_{0}}{\a}  \left|\frac{\xi}{\a}\right|^\gamma G\left(\frac{\xi}{\a}\right) + \frac{\kappa_0}{\b} \left|\frac{\xi}{\b}\right|^\gamma G\left(\frac{\xi}{\b}\right)-\kappa_{0}\,|\xi|^\gamma\, G(\xi)
\end{equation*}
and $B(\xi) = \xi G(\xi)$.  Let $\phi$ be a smooth function with support in $(-1,1)$ and satisfying $\phi(0)=1$.  For any $\varepsilon>0$, $\phi\left(\frac{\cdot}{\varepsilon}\right)$ belongs to $\mathcal{C}^1_b(\R)$, one can apply \eqref{kapphi} to get 
\begin{equation*}
\kappa_{0} \;\mathbf{m}_\gamma 
= \int_{-\varepsilon}^\varepsilon \left(A(\xi) \phi\left(\frac{\xi}{\varepsilon}\right) 
+  G(\xi) \frac{\xi}{\varepsilon}\;\phi'\left(\frac{\xi}{\varepsilon}\right)\right)\d\xi\,.
\end{equation*}
Hence,
\begin{equation*}
0 \leq \kappa_{0} \;\mathbf{m}_\gamma 
\leq \|\phi\|_{L^\infty} \int_{-\varepsilon}^\varepsilon |A(\xi)| \d\xi
+  \sup_{\xi\in\R}|\xi\,\phi'(\xi)| \int_{-\varepsilon}^\varepsilon G(\xi)
\d\xi\,.
\end{equation*}
Letting $\varepsilon \to 0$, one obtains $\kappa_{0}\,\mathbf{m}_\gamma = 0$, thus, using hypothesis \eqref{contra} we must have $\kappa_{0} = 0$.
\end{proof}
\subsection{Proof of Theorem \ref{existence}}
We have all the previous machinery at hand to prove the existence of ``physical'' solutions to \eqref{steady} in the sense of Definition \ref{defi:station} employing the dynamic fixed point Theorem \ref{GPV}.  Let us distinguish here two cases:  \medskip

$1)$ First, assume that $\gamma \in (0,1]$.  Setting $\mathcal{Y}$ to be the space $\M(\R)$ endowed with the weak-$\star$ topology, we introduce the nonempty closed convex set
\begin{equation*}
\mathcal{Z}:= \left\{\mu \in \P^0(\R)\; \text{ such that}\hspace{0.2cm}\int_{\R}|\xi|^{\gamma}\mu(\d\xi) \geq \dfrac{\gamma}{\mathcal{K}_{\gamma}}\,, \hspace{0.2cm}\text{and} \hspace{0.2cm} \int_{\R}|\xi|^{p}\,\mu(\d \xi) \leq \mathbf{M}_{p} \quad \forall\, { p \geq 2}   \right\}
\end{equation*}
where $\mathbf{M}_{p}$ was defined in Proposition \ref{prop:expo} and $\mathcal{K}_{\gamma} > 0$ is the positive constant given in Theorem \ref{theo:decmoms}.  This set is a compact subset of $\mathcal{Y}$ thanks to the uniform moment estimates (recall that $\mathcal{Y}$ is endowed with the weak-$\star$ topology).  Moreover, according to Proposition \ref{prop:expo}, there exists $\alpha > 0$ and $C(\alpha) >0$ such that, for any $\mu \in \mathcal{Z}$,
\begin{equation*}
\int_{\R}\exp(\alpha |\xi|^{\gamma})\,\mu(\d\xi) \leq C(\alpha) < \infty\,.
\end{equation*}
Thus, $\mu  \in \P_{\exp,\gamma}(\R)$.  Therefore, using Theorem \ref{theo:cauchy}, $(\mathcal{S}_{t}(\mu))_{t \geq 0}$ and, consequently,  $(\mathcal{F}_{s}(\mu))_{s\geq 0}$ are well-defined.  Setting $\nub_{s}=\mathcal{F}_{s}(\mu)$, it follows from Theorem \ref{theo:decmoms} that 
\begin{equation*}
\int_{\R}|\xi|^{p} \nub_{s}(\d\xi) \leq \mathbf{M}_{p} \qquad \forall\, p \geq 2,\, s \geq 0\,.
\end{equation*}
Using the lower bound in \eqref{eq:decmom}, we deduce that
\begin{equation*}
\int_{\R}|\xi|^{\gamma}\,\nub_{s}(\d\xi) \geq \dfrac{\gamma}{\mathcal{K}_{\gamma}} \qquad \forall s \geq 0\,.
\end{equation*}
This shows that $\nub_{s} \in \mathcal{Z}$ for any $s \geq 0$, i.e. $\mathcal{F}_{s}(\mathcal{Z}) \subset \mathcal{Z}$ for all $s \geq 0.$
Moreover, one deduces directly from Proposition \ref{prop:cont} that $(\mathcal{F}_{s})_{s}$ is continuous over $\mathcal{Z}$.  As a consequence, it is possible to apply Theorem \ref{GPV} to deduce the existence of a measure $\mub \in \mathcal{Z}$ such that $\mathcal{F}_{s}(\mub)=\mub$, a steady measure  solution to \eqref{steady}  in the sense of Definition \ref{defi:station}.  Finally, since $\mub \in \mathcal{Z}$, its moment of order $\gamma$ is bounded away from zero and by Theorem \ref{theo:meas-L1}, $\mub$ is absolutely continuous with respect to the Lebesgue measure.  This proves the result in the case $\gamma\in(0,1]$.\medskip

$2)$ Assume now $\gamma >1$ and let $\gmm=\max(\gamma,2)$.  Then, consider   $\mathcal{Y}$ to be the space $\M(\R)$ endowed with the weak-$\star$ topology and we introduce the nonempty closed convex set
\begin{equation*}
\mathcal{Z}:= \left\{\mu \in \P^0(\R)\; \text{ such that } \int_{\R}|\xi|^{p}\,\mu(\d \xi) \leq \mathbf{M}_{p}\;\; \forall\, p \geq \gmm,\;\; \text{and}\;\; \int_{\R}|\xi| \mu(\d\xi) \geq \ell\right\}
\end{equation*}
for some positive constant $\ell$ to be determined.  In fact, we prove that there exists $\ell=\ell(\gamma)$ sufficiently small such that $\mathcal{Z}$ is invariant under the semi-flow $(\mathcal{F}_{s})_{s\geq 0}$.  Indeed, according to \eqref{eq:decmom2}, for any $\ell > 0$ if $\mu_{0}$ is such that $M_{1}(\mu_{0}) \geq \ell$, then
\begin{equation*}
M_{1}\big(\mathcal{F}_{s}(\mu_{0})\big) \geq \min\bigg\{\ell\,;\,\left(\dfrac{1}{C_{\gamma,1}\,A_{1}^{1-\frac{1}{\gamma}}}\right)^{\frac{1}{\gamma}}\bigg\} \qquad \forall\,s\geq0\,,
\end{equation*}
where $C_{\gamma,1} > 0$ is some positive universal constant.  And, according to \eqref{cte}, $A_{1}$ is \emph{any} real number larger than $\max\left\{\tilde{C}_{\gamma,1}\,;\,\dfrac{M_{1+\gamma}(\mu_{0})}{M_{1}(\mu_{0})^{1+{\gamma}}}\right\}$,  where $\tilde{C}_{\gamma,1}$ is another universal positive constant.  In particular, choosing $\ell$ small enough such that 
\begin{equation*}
\dfrac{\mathbf{M}_{1+\gamma}}{\ell^{1+\gamma}} \geq \tilde{C}_{\gamma,1}\,,
\end{equation*}
it is possible to pick
$
A_{1}:=\dfrac{\mathbf{M}_{1+\gamma}}{\ell^{1+\gamma}}\,,
$
where we recall that, since $\mu_{0} \in \mathcal{Z}$ and $1+\gamma\geq\max\{\gamma,2\}$, one has $M_{1+\gamma}(\mu_{0}) \leq \mathbf{M}_{1+\gamma}$.  In such a case, one gets
\begin{equation*}
\min\bigg\{\ell\,;\,\left(\dfrac{1}{C_{\gamma,1}\,A_{1}^{1-\frac{1}{\gamma}}}\right)^{\frac{1}{\gamma}}\bigg\}=\min\Bigg\{\ell\,;\,\dfrac{\ell^{1-\frac{1}{\gamma^{2}}}}{C_{\gamma,1}^{\frac{1}{\gamma}}\,\mathbf{M}_{1+\gamma}^{\frac{\gamma-1}{\gamma^{2}}}}\Bigg\}\,.
\end{equation*}
We set $\ell \leq C_{\gamma,1}^{-{\gamma}}\,\mathbf{M}_{1+\gamma}^{1-\gamma}$ in order to get
\begin{equation*}
\min\bigg\{\ell\,;\,\bigg(\dfrac{1}{C_{\gamma,1}\,A_{1}^{1-\frac{1}{\gamma}}}\bigg)^{\frac{1}{\gamma}}\bigg\}=\ell\,,
\end{equation*}
and $M_{1}(\mathcal{F}_{s}(\mu_{0})) \geq \ell$ for any $s \geq 0$. Arguing as in the case $\gamma\in(0,1]$, this shows that $\mathcal{F}_{s}(\mathcal{Z}) \subset \mathcal{Z}$ for any $s \geq 0$ and, there exists a steady measure $\mub$ which is absolutely continuous with respect to the Lebesgue measure.
\section{Numerical simulations}\label{NS}
This section contains numerical simulations for the rescaled equation 
\begin{equation}
\label{eq:cheng2}
\partial_s g(s,\xi)-\tfrac{1}{2}\big(a^2+b^2-1\big)\partial_{\xi}\big(\xi\, g(s,\xi)\big)=\mathcal{Q}(g,g)(s,\xi), \qquad s\geq0\,,\;\; \xi \in \mathbb{R}
\end{equation}
where
$\mathcal{Q}(g,g)$ has been previously defined in \eqref{eq:cheng} and \eqref{eq:Qw}. We recall that such a model has been studied in \cite{CarrTo, pareschi} in the case of $\gamma=0$ and it admits a  \textit{unique} steady state
\begin{equation}
\label{eq:m1}
M_1(\xi)=\frac{2}{\pi} \left(\frac{1}{1+\xi^2} \right)^2,
\end{equation}
such that
\begin{equation*} \int_{\R} M_{1}(\xi)\,  \left(\begin{array}{c}1 \\ \xi \\\xi^{2}\end{array}\right)\d\xi=\left(\begin{array}{c}1 \\0 \\ 1\end{array}\right).
\end{equation*}
We will use the numerical solutions of \eqref{eq:cheng2} to verify the properties of our models for general values of $\gamma$. We shall consider here initial datum $g_{0}(\xi)=g(0,\xi)$ which shares the same first moments of $M_{1}$ i.e.
\begin{equation}
\label{eq:cond} \int_{\R} g_{0}(\xi)\,  \left(\begin{array}{c}1 \\ \xi \\ \xi^{2}\end{array}\right)\d\xi=\left(\begin{array}{c}1 \\0 \\ 1\end{array}\right).
\end{equation} The coefficient $c=-\tfrac{1}{2}\big(a^2+b^2-1\big)=ab >0$ in equation \eqref{eq:cheng2} is the only one that gives stationary self-similar profile with finite energy in the case $\gamma=0$, see \cite{CarrTo}.  In contrast, as already noticed, the case $\gamma>0$ accepts any arbitrary positive coefficient in the equation, thus, we will perform all numerical simulations with such coefficient for comparison purposes.  Recall that in our previous analysis we choose this coefficient to be $1$.  
\subsection{Numerical scheme}
To compute the solution, we have to make a technical assumption which is the truncation of $\mathbb{R}$ into a finite domain $\Omega=[-L, L]$. When $L$ is chosen large enough so that $g$ is machine zero at $\xi=\pm L$, this will not affect the quality of the solutions\footnote{Although we have no rigorous proof that the tails of $g$ will decay, at least, as fast as $Ae^{-a|\xi|^{\gamma}}$, this is expected.  For dimensions $d>1$, this is a well known fact for the elastic and inelastic Boltzmann equations.}.  Then, we use a discrete mesh consisting of $N$ cells as follows:
$$-L=\xi_{\frac{1}{2}} < \xi_{\frac{3}{2}}< \ldots < \xi_{N+\frac{1}{2}}=L\,.$$
We denote cell $I_j=\big(\xi_{j-\frac{1}{2}}, \xi_{j+\frac{1}{2}}\big)$, with cell center $\xi_j=\frac{1}{2}\big(\xi_{j-\frac{1}{2}}+\xi_{j+\frac{1}{2}}\big)$ and length $\Delta \xi_j=\xi_{j+\frac{1}{2}}- \xi_{j-\frac{1}{2}}$.
The scheme we use is  the discontinuous Galerkin (DG) method \cite{cockburn}, which has  excellent conservation properties. The DG schemes employ the approximation space defined by
$$V_h^k =\{v_h: v_h|_{I_j} \in P^k(I_j),\,1\leq j \leq N\},$$ where $P^k(I_j)$ denotes all polynomials of degree at most $k$ on $I_j$, and look for the numerical solution $g_h \in V_h^k$, such that 
\begin{multline}
\label{dg}
\int_{I_j} \partial_{s}g_h(s,\xi)\, v_h(\xi)\, \text{d}\xi+
\tfrac{1}{2}\big(a^2+b^2-1\big) \int_{I_j}\xi\, g_h(s,\xi)\, \partial_{\xi}v_h(\xi)\, \d\xi\\
+\tfrac{1}{2}\big(a^{2}+b^{2}-1\big)\left((\widehat{\xi{g}_h}\, v_h^+)_{j-\frac{1}{2}}-  (\widehat{\xi{g}_h}\, v_h^-)_{j+\frac{1}{2}} \right)\\
=\int_{I_j} \mathcal{Q}(g_h,g_h)(s,\xi)\, v_h(\xi)\, \text{d}\xi\,, \qquad j=1, \ldots, N
\end{multline}
holds true for any $v_h \in V_h^k$. In \eqref{dg}, $\widehat{\xi{g}_h}$ is the upwind numerical flux
\[ \widehat{\xi{g}_h} = \left\{ 
  \begin{array}{l l}
    \xi g_h^-(s,\xi) & \quad \text{if $\xi \ge 0$}\\
    \xi g_h^+(s,\xi) & \quad \text{if $\xi < 0$},
  \end{array}
\right.\]
where $g_h^-, g_h^+$ denote the left and right limits of $g_h$ at the cell interface.  Equation \eqref{dg} is in fact an ordinary differential equation for the coefficients of $g_h(s,\xi)$. The system can then be solved by a standard ODE integrator, and in this paper we use the third order TVD-Runge-Kutta methods \cite{shu} to evolve this method-of-lines ODE. Notice that the implementation of the collision term in \eqref{dg} is done by recalling \eqref{eq:Qw}, and we only need to calculate it for all the basis functions in $V_h^k$. This is done before the time evolution starts to save computational cost.

The DG method described above when $k\geq1$ (i.e. we use a scheme with at least piecewise linear polynomial space) will preserve mass and momentum up to discretization error from the boundary and numerical quadratures. This can be easily verified by using appropriate test functions $v_h$ in \eqref{dg}.  For example, if we take $v_h=1$ for any $j$, and sum up on $j$, we obtain
$$\int_{\Omega} \partial_{s}g_h\,\text{d}\xi=\frac{L}{2}\big(a^2+b^2-1\big) \left (g_h^-(\xi=L)+g_h^+(\xi=-L) \right ).$$
If $L$ is taken large enough so that $g_h$ achieves machine zero at $\pm L$, this implies mass conservation.  Similarly,  we can prove
$$\int_{\Omega} \partial_{s}g_h \, \xi\,\text{d}\xi=-\tfrac{1}{2}\big(a^2+b^2-1\big)\int_{\Omega} g_h\,  \xi\,\text{d}\xi+\frac{L^2}{2}(a^2+b^2-1)\left (g_h^-(\xi=L)-g_h^+(\xi=-L) \right ).$$
Again, when $L$ is large enough and the initial momentum is zero, this shows conservation of momentum for the numerical solution.  

\subsection{Discussion of numerical results}
We use as initial state the discontinuous initial profile

\[ g(0,\xi) = \left\{ 
  \begin{array}{l l}
    \frac{1}{2\sqrt{3}} & \quad \text{if $|\xi| \le \sqrt{3}$}\\
    0 & \quad \text{otherwise.}
  \end{array} \right.
\]
This profile clearly satisfies the moment conditions \eqref{eq:cond}.  We take the domain to be $[-40, 40]$ and use piecewise quadratic polynomials on a uniform mesh of size $2000$. Four sets of numerical results have been computed, corresponding to $(\gamma,a)=(1,0.1),\,(1,0.3),\,(1,0.5),\,(2,0.5)$, and $(3,0.5)$ respectively. The computation is stopped when the residual $$\sqrt{\int_{\Omega} \left (\frac{g_h(s^{n+1},\xi)-g_h(s^n,\xi)}{\Delta s} \right )^2 \text{d}\xi\,}$$ reduces to a threshold below $10^{-4}$ indicating convergence to a steady state.

In Figure \ref{fig:eq} we plot the objects of study in this document, that is, the equilibrium solutions for different values of $\gamma$. In this plot, the amplitude of the solutions has been normalized to one at the origin for comparison purposes.  The numerical solutions are used for the cases $\gamma=1, 2, 3$, while for $\gamma=0$, we use the theoretical equilibrium $M_1$ as defined in \eqref{eq:m1}.  In general terms, these smooth patterns are expected with exponential tails happening for any $\gamma>0$.  The behavior of the profiles at the origin is quite subtle and will depend non linearly on the potential, for instance, the case $\gamma=2$ renders a wider profile relative to $\gamma=0$ in contrast to $\gamma=1$ or $\gamma=3$.  This is not to say that such behavior is discontinuous with respect to $\gamma$, it is simply the net result of the contributions of short and long range interactions of the particles in equilibrium. 

In Figure \ref{fig:eqa}, we fix $\gamma=1$ and compare the stationary solution for different values of $a$.  Recall that the parameter $a$ measures the ``inelasticity'' degree of the system with $a=0$ being elastic particles and with $a=0.5$ being sticky particles.  As expected, smaller values of $a$ will render a wider distribution profile at the origin keeping the tails unchanged.  Near the origin, the distribution of particles for less inelastic systems will be underpopulated relative to more inelastic systems which force particles to a more concentrated state.  Tails, however, are more dependent to the growth of the potential and should remain relatively unchanged despite changes in inelasticity.  Interestingly, the numerical simulation shows a unexpected effect:  the maximum density of particles is not necessarily located at the origin.

\begin{figure}[!htbp]
\centering
\includegraphics[width=0.65\textwidth]{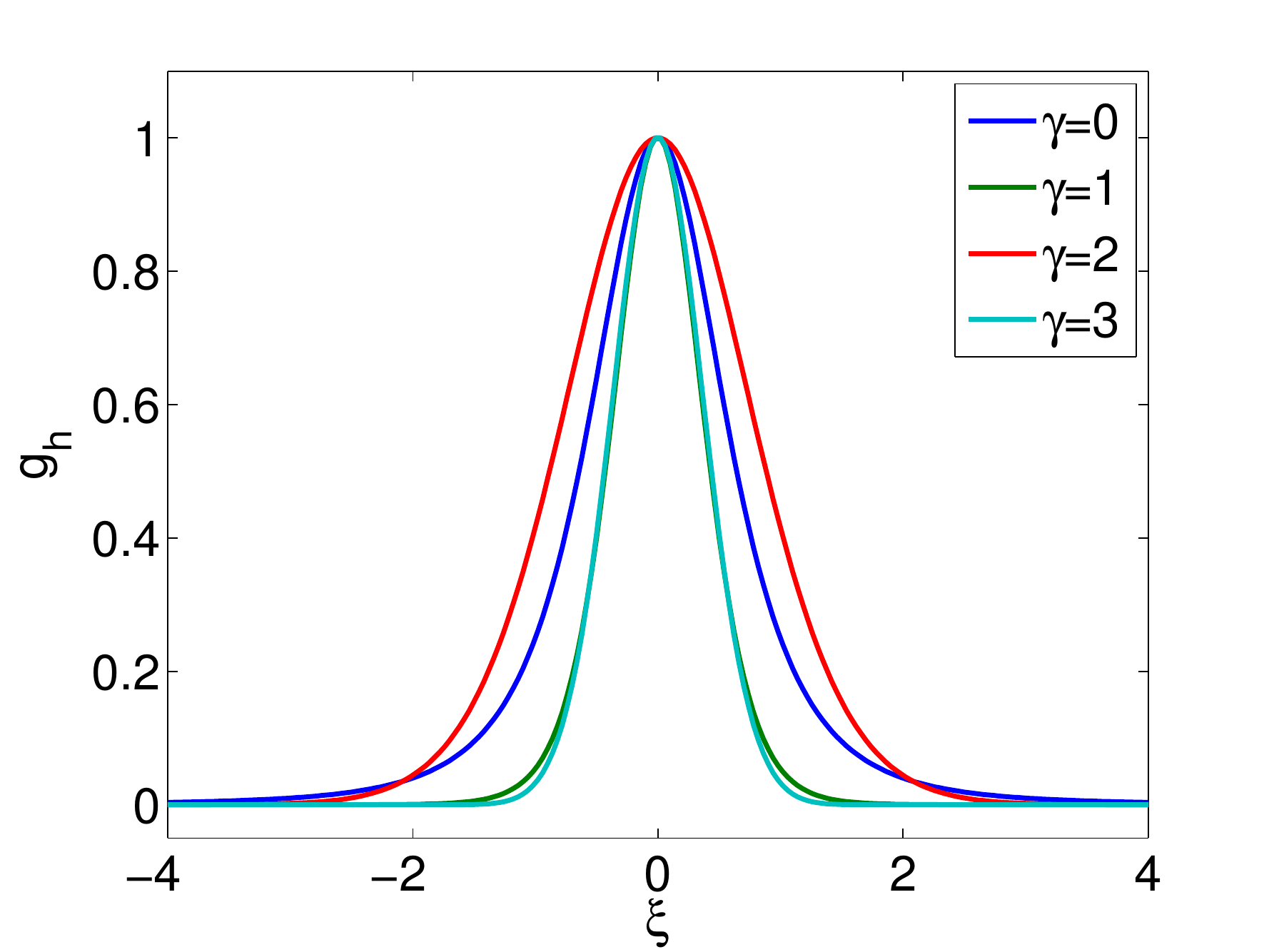}
\caption{Rescaled equilibrium solutions for different values of $\gamma$ with $a=0.5$ (sticky particles).  Curves corresponding to $\gamma=1, 2, 3$  are computed numerically, while the curve for $\gamma=0$ is obtained from the known steady state $M_1$ in \eqref{eq:m1}.}
\label{fig:eq}
\end{figure}

\begin{figure}[!htbp]
\centering
\includegraphics[width=0.65\textwidth]{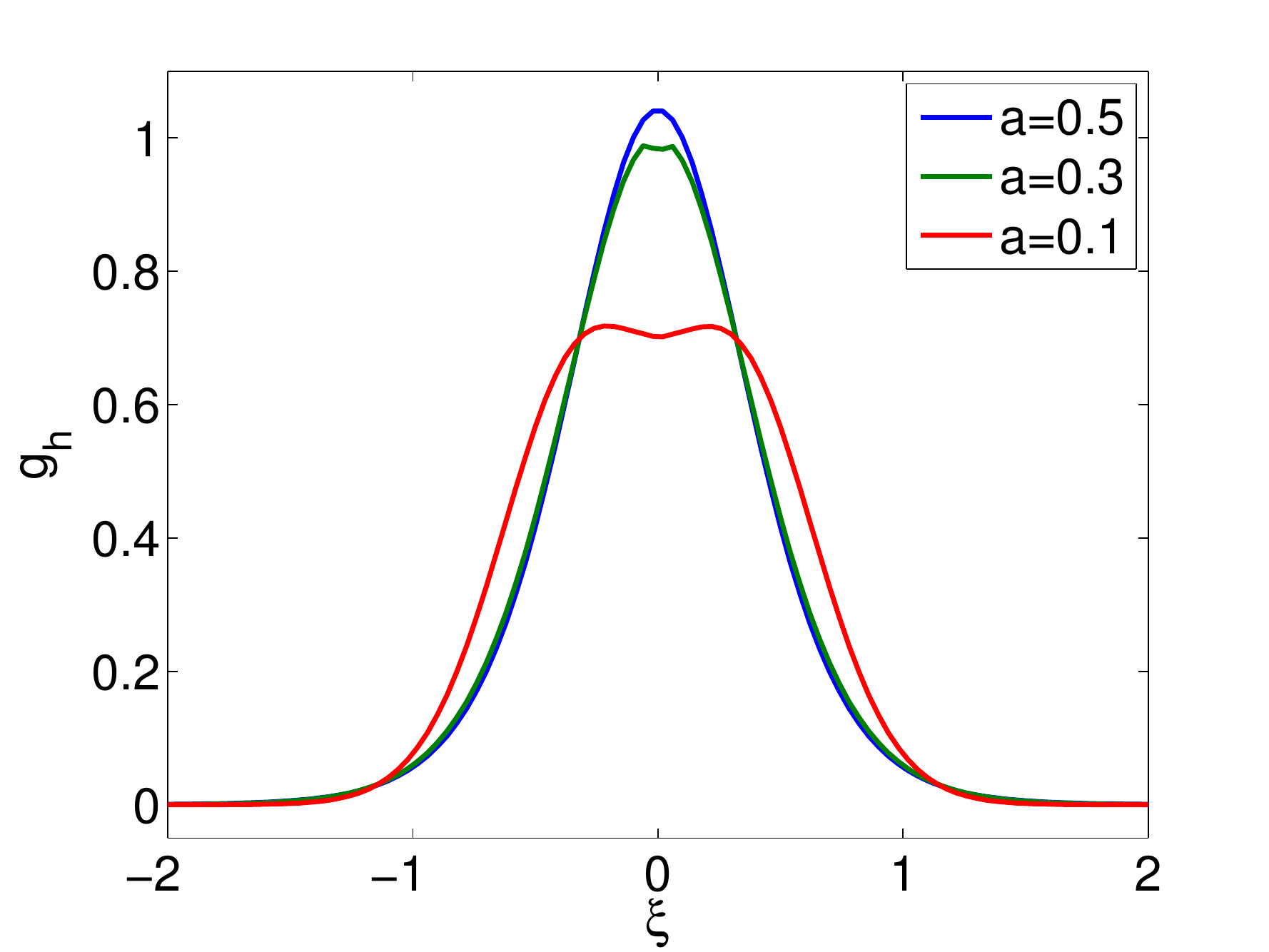}
\caption{Equilibrium solutions for different values of inelasticity $a$ when $\gamma=1$.}
\label{fig:eqa}
\end{figure}

In Figure \ref{fig:ene}, we plot the evolution of energy as a function of time in a system of sticky particles $a=0.5$ using different values of $\gamma$.  Changes in the relaxation times are expected since the potential growth $\gamma$ impacts directly on the spectral gap of the linearized interaction operator.  This numerical result seems to confirm, in one dimension, the natural idea that higher $\gamma$ implies higher spectral gap, hence, faster relaxation to equilibrium.  Refer to \cite{MiMo2} for ample discussion in higher dimensions for the so called quasi-elastic regime.  Additionally, the results of Figure \ref{fig:ene} are numerical confirmation of the optimal cooling rate given in our Theorems \ref{theo:decmom} and \ref{theo:decmoms}.

\begin{figure}[!htbp]
\centering
\includegraphics[width=0.65\textwidth]{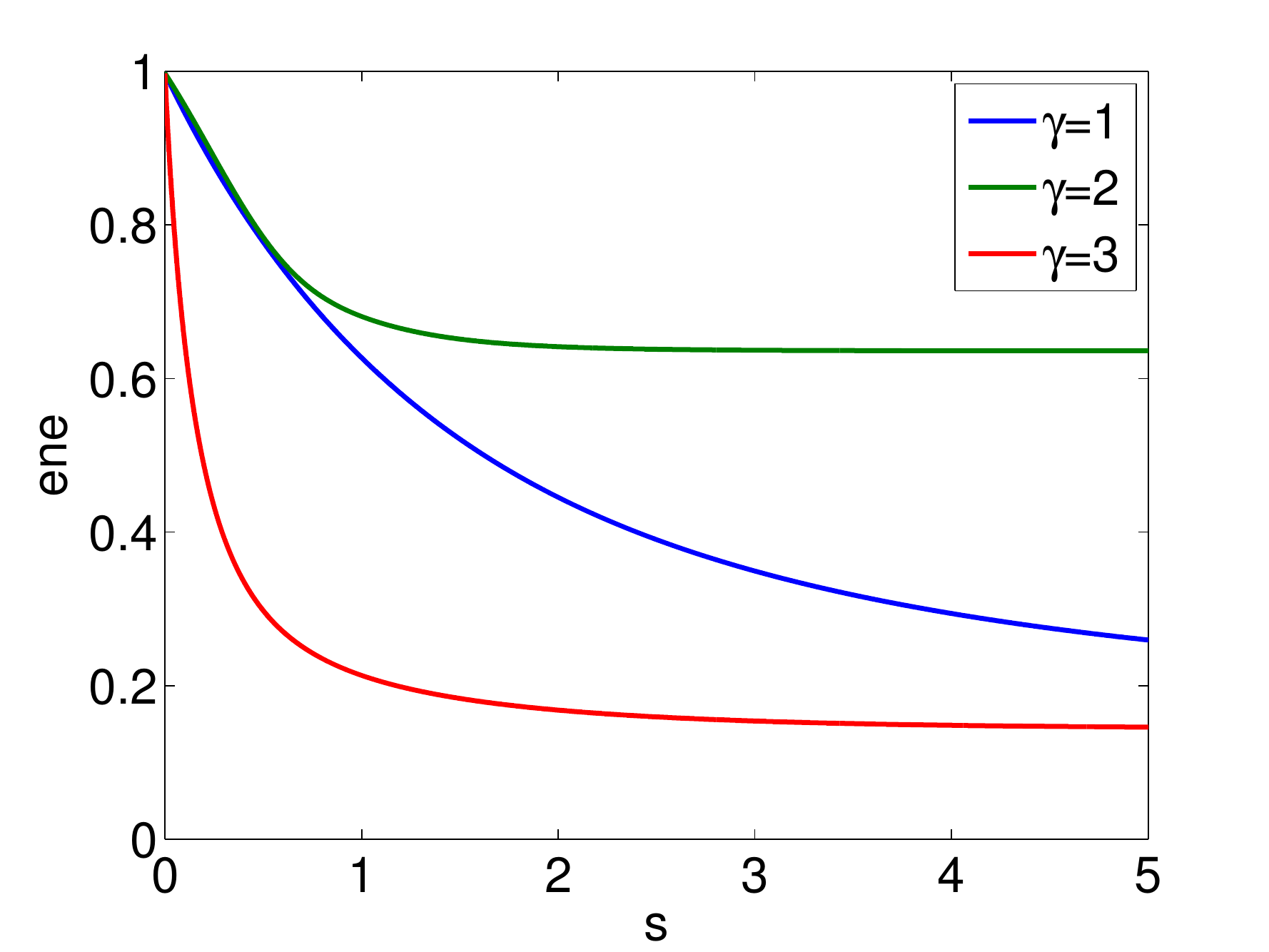}
\caption{Evolution of energy for $\gamma=1, 2, 3$, when $a=0.5$. }
\label{fig:ene}
\end{figure}

In Figure \ref{fig:soln}, we investigate the evolution of the distribution function and its discontinuities for the case $\gamma=1$ and $a=0.5$.  The simulation shows, in our one dimension setting, a well stablished phenomena happening in elastic and quasi-elastic Boltzmann equation in higher dimensions:  discontinuities are damped at exponential rate \cite{MiMo2}.  As a consequence, points of low regularity which are contributed by $\mathcal{Q}^{+}(g,g)$ due to such discontinuities will be smoothed out exponentially fast as well.  This is the case for the point $\xi=0$ in this particular simulation.  A numerical simulation was also performed using an initial Gaussian profile (not included).  Both numerical simulations showed an evolution towards the same equilibrium profile which confort us in the belief that the constructed self-similar profile is unique.

\begin{figure}[!htbp]
\centering
\includegraphics[width=0.65\textwidth]{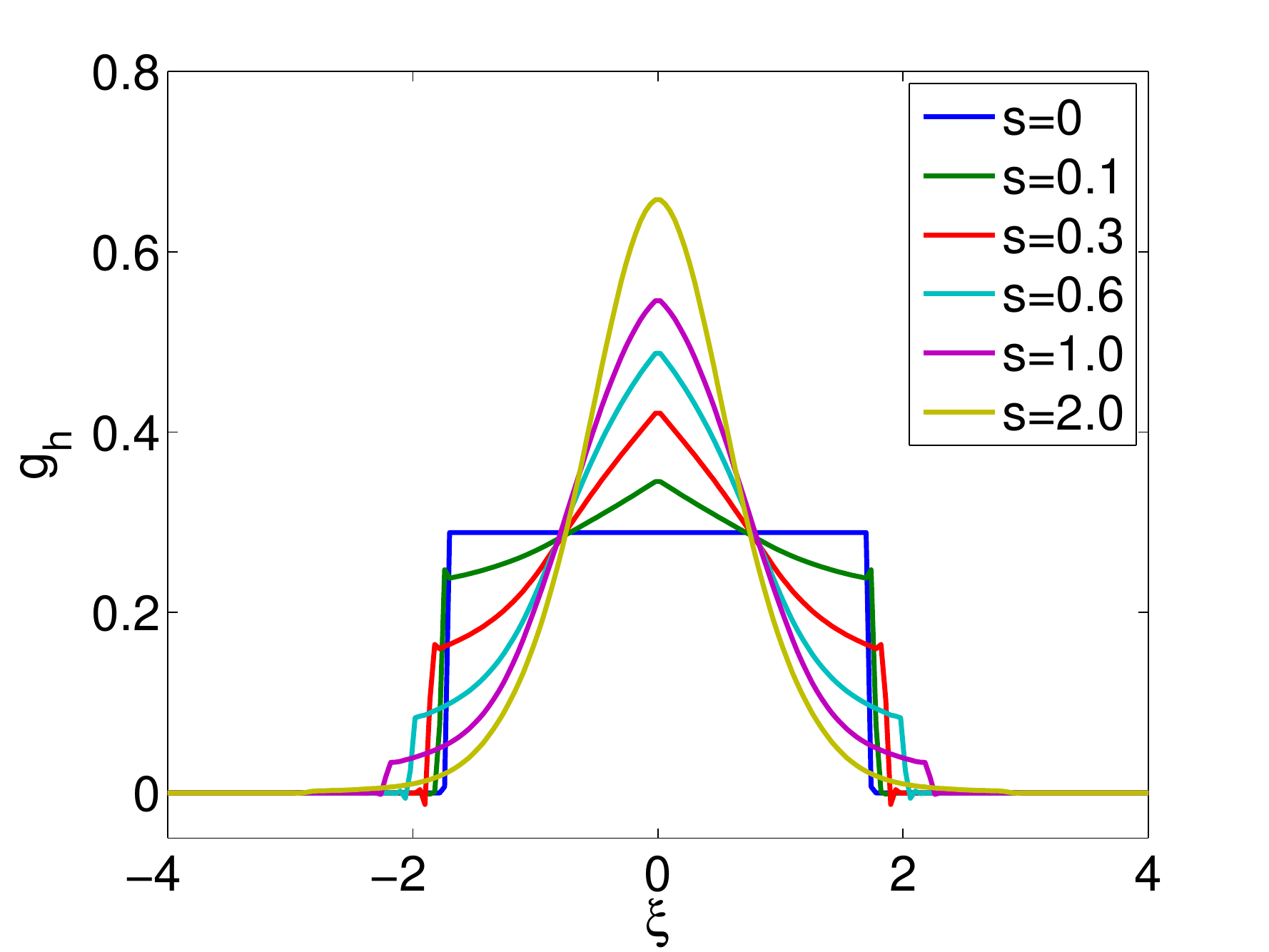}
\caption{Evolution of $g(s,\xi)$ when $\gamma=1, a=0.5$. }
\label{fig:soln}
\end{figure}
Finally, we verify the performance of our scheme by plotting the distribution's mass, momentum and energy.  Only the plot for $\gamma=1, a=0.5$ is shown since the other cases display similar accuracy.    In Figure \ref{figure_g1}  we plot the evolution of mass, momentum, energy and residual (in the log scale). The decay of residual shows convergence to steady state, while mass and momentum are preserved up to 10 digits of accuracy verifying the performance of the DG method. 

\begin{figure}[!htbp]
\centering
\subfigure{\includegraphics[width=0.45\textwidth]{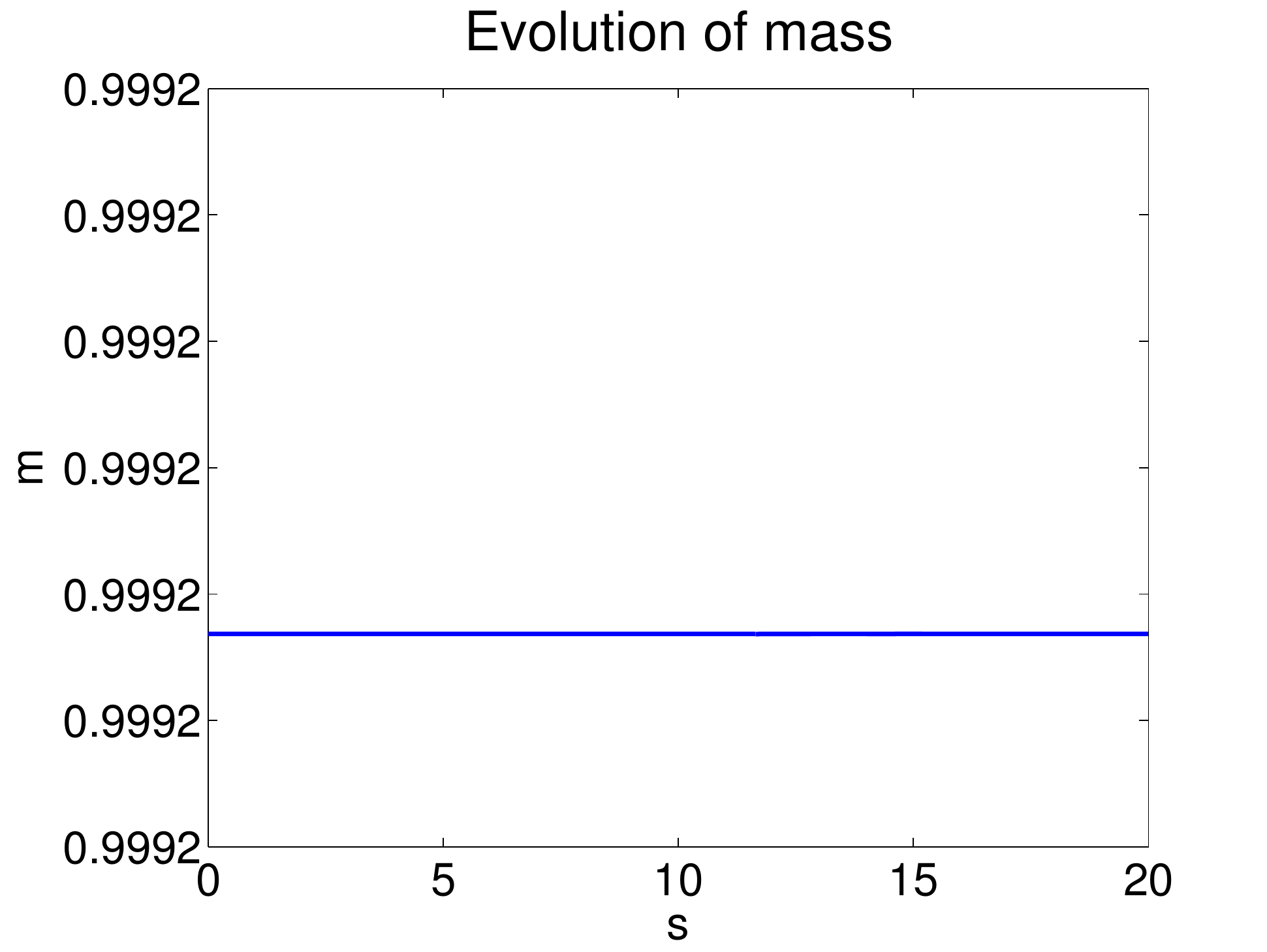}}
\subfigure{\includegraphics[width=0.45\textwidth]{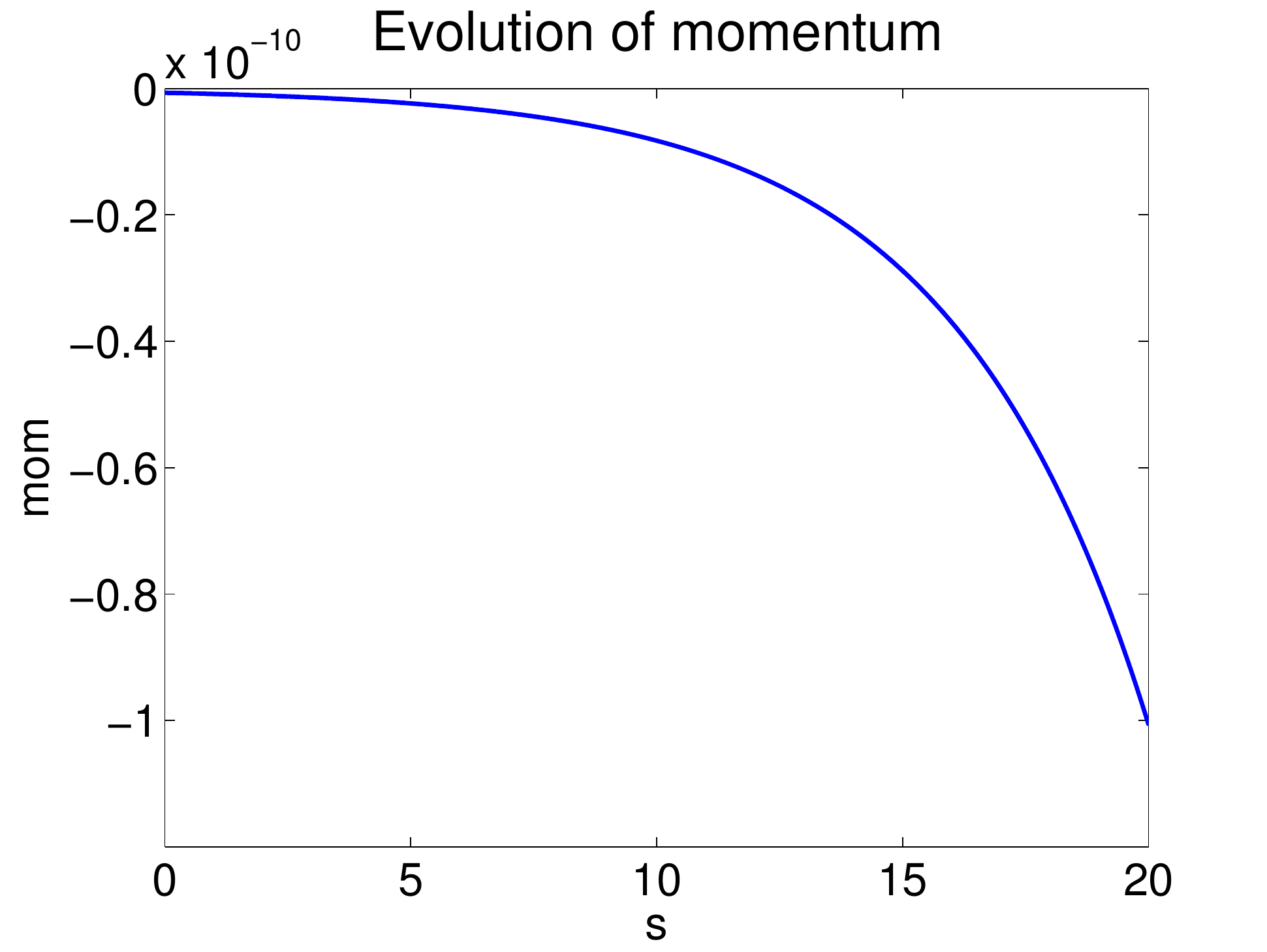}}
\subfigure{\includegraphics[width=0.45\textwidth]{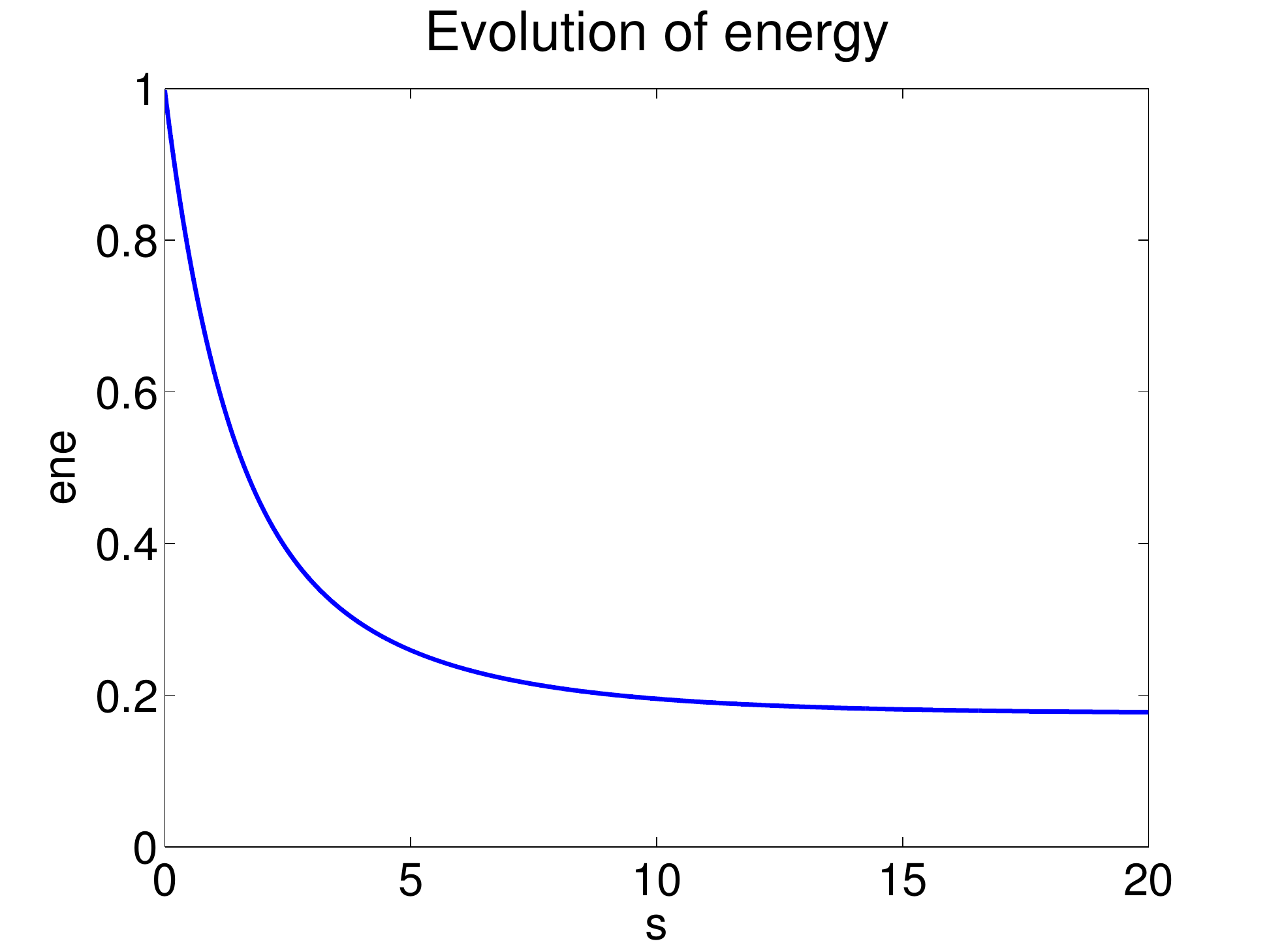}}
\subfigure{\includegraphics[width=0.45\textwidth]{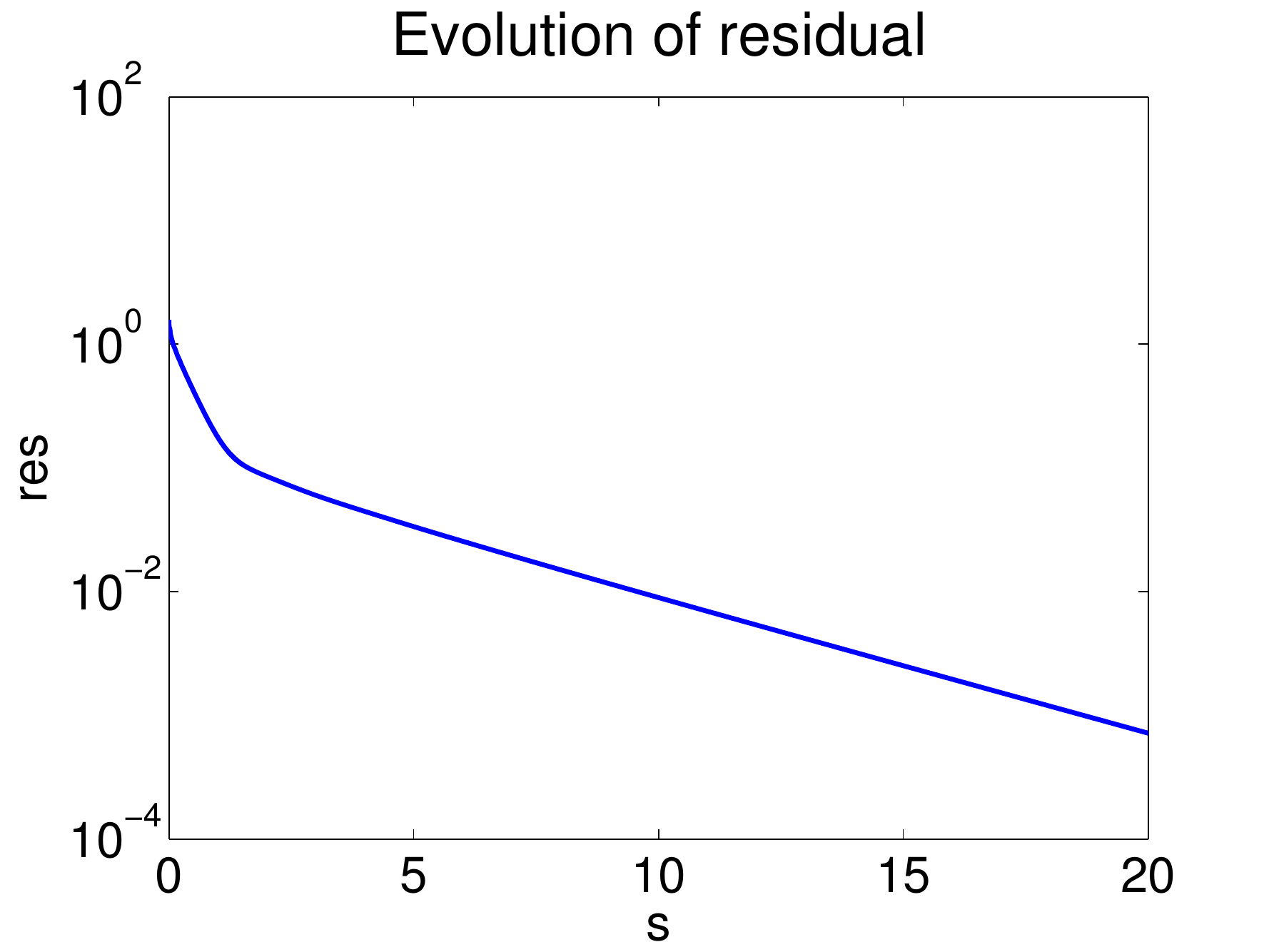}}
\caption{$\gamma=1$, $a=0.5$, evolution of mass, momentum, energy and residual. }
\label{figure_g1}
\end{figure}

\section{Conclusion and perspectives}
In the present paper we studied the large time behavior of the solution to the dissipative Boltzmann equation in one dimension.  The main achievement of the document is threefold: (1) give a proof for the well-posedness of such problem in the measure setting, (2) provide a careful study of the moments, including the optimal rate of convergence of solutions towards the Dirac mass at $0$ in Wasserstein metric, and (3) prove the existence of ``physical'' steady solutions in the self-similar variables, that is, steady measure solutions that are in fact absolutely continuous with respect to the Lebesgue measure.  Let us make a few comments about the perspectives and related open problems.
\subsection{Regularity propagation for inelastic Boltzmann in 1-D}  The numerical simulations performed in Section \ref{NS} seem to confirm that many of the known results given for inelastic Boltzmann in higher dimensions should extend to inelastic Boltzmann in 1-D, at least, under suitable conditions.  More specifically, rigorous results about propagation of Lebesgue and Sobolev norms, and exponential attenuation of discontinuities for the time evolution problem should hold.  Similarly, the study of optimal regularity for the stationary problem is an interesting aspect of the equation which is unknown. 
\subsection{Alternative approach \textit{à la} Fournier-Lauren\c cot} Exploiting the analogy between \eqref{gs} and the self-similar Smoluchowski's equation, one may wonder if the approach performed by N. Fournier \& Ph. Lauren\c cot in \cite{Fournier} can be adapted to \eqref{steady}.  We recall that the approach in \cite{Fournier} consists in finding a suitable discrete approximation of the steady problem for which a discrete steady solution can be constructed. If such discrete solution exhibits all the desired properties (positivity, uniform upper bounds and suitable lower bounds) uniformly with respect to the discretization parameter, then, one can pass to the limit to obtain the desired steady solution to \eqref{steady}.  Such approach fully exploits the 1-D feature of the problem.  Besides, it does not resort to the evolution equation \eqref{gs}, fact that makes it very elegant. The main contrast with respect to \cite{Fournier} lies in the fact that no estimates for moments of \emph{negative} order seem available for our problem.  Moreover, Smoluchowski's equation is such that the collision-like operator sends mass to infinity while the drift term brings it back to zero.  The model \eqref{steady} has the opposite behavior: the collision tends to concentrate mass in zero while the drift term sends it to infinity.
\subsection{Uniqueness and stability of the self-similar profile} Now that the existence of a steady solution to \eqref{steady} has been settled, the next challenge is to prove that such self-similar profile is unique and that it attracts solutions to \eqref{gs} as $s \to \infty$ or, at least, to find conditions for this to hold.  This is certainly the case in the simulations performed in Section \ref{NS} which show, in addition, exponential rate of attraction.  For the 3-D inelastic Boltzmann equation, such a result has been proven in \cite{MiMo2} in the so-called weakly inelastic regime (a perturbation of the elastic problem).  Since the 1-D Boltzmann equation is meaningless for elastic interactions a perturbative approach seems inadequate.  Once a uniqueness theory is at hand, it would be desirable to obtain rate of convergence, see for instance \cite{AL3}.  This would render a more complete picture of the large time behavior of the dissipative Boltzmann equation on the line.
\subsection{The rod alignment problem by Aranson and Tsimring}\label{sec:intro-align}
Aranson and Tsimring  in \cite{AT} have introduced the following model for rod alignment (the rods have distinguishable beginning and end)
\begin{equation}\label{eq:rod}
\partial_{t}P(t,\theta)=\ds\int^{\pi}_{-\pi}P\Big(t,\theta - \frac{\theta_*}{2}\Big)\,P\Big(t,\theta+\frac{\theta_*}{2}\Big)\,\big|\theta_*\big|^{\gamma}\,\d\theta_*  - P(t,\theta)\,\ds\int^{\pi}_{-\pi}P(t,\theta + \theta_*)\,\big|\theta_*\big|^{\gamma}\,\d\theta_{*}
\end{equation}
having initial condition $P(0)=P_{0}$, angle $\theta\in(-\pi,\pi)$.  The authors introduced the model for Maxwellian interactions $\gamma=0$, yet, the model is sound for any $\gamma\geq0$.  We refer to \cite{Ben-Naim,carlen} for other variations of such model.  Here $P(t,\theta)$ is the time--distribution of rods having orientation $\theta\in[-\pi,\pi)$.  Equation \eqref{eq:rod} models a system of many discrete rods aligning by the pairwise  irreversible law
\begin{equation}\label{il}
\big(\theta-\tfrac{\theta_*}{2},\theta+\tfrac{\theta_*}{2})\rightarrow\big(\theta,\theta\big).
\end{equation}
Let us explain the interaction law \eqref{il}.  We start by fixing a horizontal frame and picking two interacting rods with orientation $\theta_{1},\theta_{2}\in[-\pi,\pi)$.  Define $\theta_* \in [-\pi,\pi)$ as the angle between the ends of the rods.  Bisect the rods and define $\theta\in[-\pi,\pi)$ as the angle between the horizontal frame and the bisecting line.  Thus, we can express the rods orientation, up to modulo $2\pi$, by the relation $\theta_{1} = \theta-\tfrac{\theta_*}{2}$ and $\theta_2 = \theta+\tfrac{\theta_*}{2}$ with respect to the horizontal frame.   After interaction, both rods will align with the bisection angle $\theta$.  This law produces the alignment of rods, we refer to \cite{AT, Ben-Naim} for an interesting discussion and simulations.  The law \eqref{il} can be written in terms of the rod orientations $\theta_{1},\,\theta_{2}\in[-\pi,\pi)$ as
\begin{equation}\label{ila}
\big(\theta_{1},\theta_{2}\big)\rightarrow
\Bigg\{
\begin{array}{cc}
\big(\tfrac{\theta_{1}+\theta_2}{2},\tfrac{\theta_{1}+\theta_2}{2}\big) & \big|\theta_{1} - \theta_{2} \big| \leq \pi,\\
\big(\tfrac{\theta_{1}+\theta_2}{2}+\pi,\tfrac{\theta_{1}+\theta_2}{2}+\pi\big) & \big|\theta_{1} - \theta_{2} \big| > \pi.
\end{array}
\end{equation}
Note that in the case $\big|\theta_{1} - \theta_{2} \big| > \pi$ the addition of $\pi$ is needed since we chose the alignment to occur in the direction of the bisecting angle associated to the ends of the rods (as opposed to the beginnings of the rods).  The interaction law \eqref{ila} is discontinuous, thus, intuitively we understand that model \eqref{eq:rod} will not have conservation of momentum because there is a choice of alignment direction.  Let us fix this by considering an initial datum $P_{0}$ with compact support in $(-\pi/2,\pi/2)$
\begin{equation*}
\mathrm{Supp} P_{0} \subset (-\pi/2,\pi/2)\,.
\end{equation*}
Such a property is conserved by the dynamic of \eqref{eq:rod} and it corresponds to a system of rods where \textit{rod's beginning and end are indistinguishable}, thus, we can always assign an angle $\theta\in(-\pi/2,\pi/2)$ to each rod.  For such a model, the weak-formulation is very similar to that of \eqref{eq:1} except for the fact that all integrals are considered now over the finite interval $(-\pi/2,\pi/2)$.  For this reason, the decay of the moments of the solution $P(t)$ to \eqref{eq:rod} is identical to that of \eqref{eq:1}.  Consequently, this translates into the convergence of $P(t)$ towards a Dirac mass centered at $0$ as $t \to \infty$ in the Wasserstein metric.  The question is to understand the model after self-similar rescaling where the support of solutions is no longer fixed and given by $\big(-V(t)\pi/2,V(t)\pi/2\big)\rightarrow(-\infty,\infty)$ as $t\rightarrow\infty$.  Thus, it is natural to expect that the self-similar solution to \eqref{eq:rod} will converge towards the steady solution to \eqref{steady}.
\subsection{Extension to other collision-like problems} It seems that the present approach is robust enough to be applied to various contexts. In particular, the argument may be helpful to tackle notoriously difficult questions, such as, the existence of a stationary self-similar solution to the Smoluchowski equation with \textit{ballistic kernel interactions}.  It may be possible, also, to give a more natural treatment of the stationary inelastic Boltzmann equation in the framework of probability measures.  The difficulty will be to find a dynamical stable set in order to apply the dynamical fixed point and a suitable regularization theory for the stationary equation of the particular problem.
\appendix
\section{Cauchy theory in both the $L^{1}$-context and the measure setting}\label{sec:cauchyA}
In this Appendix, we give a detailed proof of the existence and stability estimates of Section \ref{sec:steps} yielding to Theorem \ref{theo:cauchy}. We fix here $\gamma > 0$ and set $\gmm=\max(\gamma,2).$ We begin with an existence and uniqueness result for the Cauchy problem \eqref{eq:1} in the special case in which the initial datum $\mu_{0}$ is absolutely continuous with respect to the Lebesgue measure, i.e.
\begin{equation*}
\mu_{0}(\d x)=f_{0}(x)\d x\,.
\end{equation*}
\begin{theo}\label{theo:L1}
Fix $\delta>0$. Let a nonnegative $f_{0}\in L^1_{2+\gamma+\delta}(\R)$  be given with $f_{0} \neq 0$. 
Setting $\mu_{0}(\d x)=f_{0}(x)\d x$, there exists a \emph{unique} family $(f_{t})_{t \geq 0} \subset  L^{1}_{2+\gamma+\delta}(\R)$ such that $\mu_{t}(\d x)=f_{t}(x)\d x$ is a weak measure solution to \eqref{eq:1} associated to $\mu_{0}$.  Moreover,
\begin{align}\label{cons_L1}
\begin{split}
\int_{\R} &f_{t}(x)\d x=\int_{\R} f_{0}(x)\d x, \qquad \int_{\R} x f_{t}(x)\d x=\int_{\R} xf_{0}(x)\d x,\\
&\text{and} \quad \int_{\R} |x|^{2} f_{t}(x)\d x\leq \int_{\R} |x|^{2} f_{0}(x)\d x \qquad \forall\, t \geq 0\,.
\end{split}
\end{align}
In addition to this, if one assumes that 
$$f_{0} \in  \bigcap_{k \geq 0} L^{1}_{k}(\R)\,,$$
then $(f_{t})_{t \geq 0} \subset  \bigcap_{k \geq 0} L^{1}_{k}(\R)$.  
\end{theo}
\begin{proof}
We follow here the approach of some unpublished notes of  Bressan \cite{bressan}. For $K>0$ and $\delta>0$, we introduce 
$$\Omega_{K,\delta}= \left\{0\leq f\in L^1(\R), \quad \int_\R f(x)\, \d x= 
 \int_\R f_0(x)\, \d x, \quad \int_\R f(x) |x|^{2+\gamma+\delta}\d x\leq K\right\}\,.$$
For $f\in\Omega_{K,\delta}$, changing variables in the collision operator leads to 
\begin{align*}
\|\Q(f,f)\|_{L^1_2} \leq \int_\R & \int_\R  f(x)f(y)|x-y|^\gamma \left(1+(bx+ay)^2\right)\d x\, \d y\\
& + \int_\R \int_\R f(x)f(y)|x-y|^\gamma \left(1+x^2\right)\d x\, \d y\,.
\end{align*}
Now, for any $(x,y)\in\R^2$, 
$$|x-y|^\gamma \leq C_\gamma \big(|x|^\gamma+|y|^\gamma\big)$$ 
with $C_{\gamma}=\max\{1;2^{\gamma-1}\}$. Therefore, one observes that
$$
\|\Q(f,f)\|_{L^{1}_{2}} \leq 4C_{\gamma}\,M_{\gamma}(f)M_{0}(f) + 2C_{\gamma}\,M_{2}(f)\,M_{\gamma}(f)+2C_{\gamma}\,M_{2+\gamma}(f) M_{0}(f)\,,
$$
where $M_{k}(f)=\int_{\R}f(x)|x|^{k}\d x$ for any $k \geq 0$.  Notice that $M_{0}(f)=\|f_{0}\|_{L^{1}}$.  Using Young's inequality, for any $0 < k < 2+\gamma +\delta$ one has 
$M_{k}(f) \leq K+\|f_{0}\|_{L^{1}}$ from which we deduce that there exists $C > 0$ such that
$$\|\Q(f,f)\|_{L^1_2}  \leq  C\left(\|f_0\|_{L^1}^2 +K^2\right)\,.$$
Consequently, $\Q(f,f)\in L^1_2(\R)$. Let us now prove that the restriction to $\Omega_{K,\delta}$ of the mapping $f \mapsto \Q(f,f)$ is H\"older continuous. For $f,g\in\Omega_{K,\delta}$, 
\begin{align*}
\|\Q(f,f)-\Q(g,g)\|_{L^1_2}  \leq   &\int_\R  \int_\R |(f-g)(x)|f(y)|x-y|^\gamma (1+(bx+ay)^2)\d x\, \d y\\
&+ \int_\R \int_\R g(x)|(f-g)(y)||x-y|^\gamma (1+(bx+ay)^2)\d x\, \d y\\
&+  \int_\R \int_\R |(f-g)(x)|f(y)|x-y|^\gamma \left(1+x^2\right)\d x\, \d y\\
&+  \int_\R \int_\R g(x)|(f-g)(y)||x-y|^\gamma \left(1+x^2\right)\d x\, \d y\,.
\end{align*}
Proceeding as previously, one notices that there exists $C > 0$ such that  
$$\|\Q(f,f)-\Q(g,g)\|_{L^1_2} \leq C( \|f_0\|_{L^1}+K) \|f-g\|_{L^1_2}
+ 2C_\gamma \|f_0\|_{L^1}\int_{\R}  |(f-g)(x)| |x|^{2+\gamma}\d x\,.$$
Thanks to the H\"older inequality, we have 
\begin{align*}
\int_{\R}  |(f-g)(x)| |x|^{2+\gamma}\d x & \leq \left(\int_{\R}  |(f-g)(x)| |x|^{2+\gamma+\delta}\d x\right)^{\frac{\gamma}{\gamma+\delta}}  \left(\int_{\R}  |(f-g)(x)| |x|^{2}\d x\right)^{\frac{\delta}{\gamma+\delta}} \\
& \leq (2K)^{\frac{\gamma}{\gamma+\delta}} \|f-g\|_{L^1_2}^{\frac{\delta}{\gamma+\delta}}\,.
\end{align*}
Combining the previous two inequalities, we deduce that the mapping $f\mapsto \Q(f,f)$ is uniformly H\"older continuous on $L^1_2(\R)$ when restricted to $\Omega_{K,\delta}$.  Let us look for a one-sided Lipschitz condition.  For $f,g\in L^1_2(\R)$, we introduce 
$$ \big[f,g\big]_-=\lim_{s\to 0^-}\frac{\|f+sg\|_{L^1_2} -\|f\|_{L^1_2}}{s}\,.$$
The dominated convergence theorem implies that 
$$\big[f,g\big]_-\leq \int_{\R} \mbox{sign}(f(x)) g(x)(1+x^2)\, \d x\,.$$
Our aim is to show that there exists a constant $L>0$ such that for any $f,g\in \Omega_{K,\delta}$, 
$$\big[f-g, \Q(f,f)-\Q(g,g)\big]_-\leq L \|f-g\|_{L^1_2}\,.$$
But, 
\begin{align*} 
\int_{\R} \mbox{sign}\big((f-g)&(x)\big)\big(\Q(f,f)-\Q(g,g)\big)(x)(1+x^2)\, \d x \\
&\hspace{-.3cm}\leq\frac{1}{2} \int_{\R} \int_{\R} \left|(f-g)(x-ay)\right|(f+g)(x+by) |y|^\gamma(1+x^2)\, \d x\, \d y\\
&+\frac{1}{2} \int_{\R} \int_{\R} (f+g)(x-ay)\left|(f-g)(x+by)\right| |y|^\gamma(1+x^2)\, \d x\, \d y\\
&+ \frac{1}{2} \int_{\R}\int_{\R}  \left|(f-g)(x+y)\right| (f+g)(x) |y|^\gamma (1+x^2)\, \d x\, \d y\\
&-\frac{1}{2} \int_{\R}\int_{\R} |(f-g)(x)| (f+g)(x+y) |y|^\gamma (1+x^2)\, \d x\, \d y\,.
\end{align*}
Thus, changing variables leads to 
\begin{align*}
\big[f-g,\,&\Q(f,f)-\Q(g,g)\big]_- \\
&\hspace{-.3cm}\leq \frac{1}{2} \int_{\R} \int_{\R}|(f-g)(x)| (f+g)(y) |x-y|^\gamma \left(3+(ax+by)^2 +(bx+ay)^2+y^2\right)\, \d x\, \d y\\
&-\frac{1}{2} \int_{\R}\int_{\R} |(f-g)(x)| (f+g)(y) |x-y|^\gamma (1+x^2)\, \d x\, \d y\,.
\end{align*} 
Finally, we obtain 
\begin{align*}
[f-g, \Q(f,f)-\Q(g,g)]_- &\leq\int_{\R} \int_{\R}|(f-g)(x)| (f+g)(y) |x-y|^\gamma (1+y^2)\, \d x\, \d y\\
& \leq L\|f-g\|_{L^1_2}\,,
\end{align*}
with $L=2 C_\gamma (4\|f_0\|_{L^1}+3K)$.  Next, let us look for a sub-tangent condition. Given $f\in\Omega_{K,\delta}$ and $h \geq 0$, one notices that
%$0\leq h\leq \dfrac{1}{\|f_0\|_{L^1}+K}$, it follows that $f+h\Q(f,f)\in\Omega_{K,\delta}$.  Indeed, given $h\in\R$
$$f(x)+h\Q(f,f)(x) =h\int_\R f(x-ay)f(x+by)|y|^\gamma\d y+f(x)\left(1-h \int_\R f(x+y)|y|^\gamma\d y\right)\,.$$
In particular, what prevents $f+h\Q(f,f)$ to be a.e. nonnegative is the influence of large $x$ in the last convolution integral. To overcome this difficulty, for any $R > 0$, we introduce the truncation $f_{R}(x)=f(x)\chi_{\{|x|<R\}}.$ Then, since $f \geq f_{R}$ one deduces from the above identity that
\begin{align}\label{eq:hfR}
\begin{split}
f(x)+h\Q(f_{R},f_{R})&(x)\geq h\int_\R f_{R}(x-ay)f_{R}(x+by)|y|^\gamma\d y\\
&+f_{R}(x)\left(1-h \int_\R f_{R}(x+y)|y|^\gamma\d y\right) \quad \text{ a.e. } x \in \R\,.
\end{split}
\end{align}
Now, 
$$\int_{\R} f_{R}(x+y)|y|^\gamma\d y=\int_\R f_{R}(y)|x-y|^\gamma\d y\leq \max(2^{\gamma-1},1)\int_{\R}f(y)\left(|x|^{\gamma}+|y|^{\gamma}\right)\d y$$
and using Young's inequality, one sees that there exists some positive constant $C_{0}>0$ depending only on $K,\delta,\gamma$ and $\|f_{0}\|_{L^{1}}$ but not on $R$ such that
$$\int_\R f_{R}(x+y)|y|^\gamma\d y
\leq C_{0}(1+|x|^{\gamma}) \qquad \text{ for any  } x \in \R\,,\,  R > 0 \text{ and }  f \in \Omega_{K,\delta}.$$
Therefore, recalling that $f_{R}$ is supported on $\{|x|<R\}$, one deduces from \eqref{eq:hfR} that
$$f(x)+h\Q(f_{R},f_{R})(x) \geq 0 \qquad \text{ for a.e. } x \in \R, \quad \forall 0 < h < h_{R}:=\dfrac{1}{C_{0}(1+R^{\gamma})}\,.$$
Moreover, since $\Q$ preserves the mass, 
$$\int_\R \big(f(x)+h\Q(f_{R},f_{R})(x)\big)\d x=\int_\R f(x)\d x=\int_\R f_0(x)\d x\,.$$
Finally, using \eqref{elem} it follows that, for any $R >0$
\begin{align*}
\int_\R \Q(f_{R},f_{R})&(x)|x|^{\gamma+2+\delta}\d x\\
&\leq - \frac{1}{2} \left(1-a^{\gamma+2+\delta}-b^{\gamma+2+\delta}\right) \int_\R\int_\R f_{R}(x)f_{R}(y) |x-y|^{2\gamma+2+\delta} \d x \d y \leq 0\,.
\end{align*}
Consequently, 
$$\int_\R  \big(f(x)+h\Q(f_{R},f_{R})(x)\big)|x|^{\gamma+2+\delta}\d x \leq \int_\R f(x)|x|^{\gamma+2+\delta}\d x  \leq K\, \qquad \forall R > 0\,.$$
We have thus shown that, for any $R > 0$ and any  $0 < h < h_{R}$, one has
$f+h\Q(f_{R},f_{R})\in\Omega_{K,\delta}$. In particular, for any $R >0$ and any $0 < h< h_{R}$ one has
\begin{align*}
\mathrm{dist}&(f+h\Q(f,f),\Omega_{K,\delta})\\
&\leq \|f+h\Q(f,f)-(f+h\Q(f_{R},f_{R})\|_{L^{1}}=h\,\|\Q(f,f)-\Q(f_{R},f_{R})\|_{L^{1}}\,.
\end{align*}
Now, for $f \in \Omega_{K,\delta}$, one can make $\|\Q(f,f)-\Q(f_{R},f_{R})\|_{L^{1}}$ arbitrarily small provided $R >0$ is large enough so that the sub-tangent condition
$$\liminf_{h\to 0^+}\,h^{-1} \mbox{dist}\left( f+h\Q(f,f),\Omega_{K,\delta}\right)=0$$
holds true.
We may now apply \cite[Theorem VI.4.3]{Martin} and deduce the existence and the uniqueness of a global solution $f$ to \eqref{eq:1} such that $f(t)\in\Omega_{K,\delta}$ for every $t\geq0$.  Moreover, \eqref{cons_L1} holds and, if $f_0\in\bigcap_{k\geq 0}L^1_k(\R)$, it follows from \eqref{elem} that for every   $k\geq \gmm$ and every $t\geq 0$, 
$$\int_\R f(t,x)\,|x|^k \d x\leq \int_\R f_0(x)\,|x|^k \d x,$$ 
which implies (together with the conservation of the mass) that $f(t) \in L^1_k(\R)$ for any $t\geq0$ and any $k \in\R$. Finally, it is easily checked that the family $(\mu_t)_{t\geq 0}$ defined by $\mu_t(\d x)=f(t,x)\d x$ for any $t\geq 0$ is a weak measure solution to \eqref{eq:1}. \end{proof}

\begin{proof}[Proof of Proposition \ref{prop:existence}] The proof follows the approach of  \cite[Section 4]{Lu} and we only sketch the main steps of the proof. First, since $\mu_{0} \in \P_{\gmm}^{0}(\R)$ is not the Dirac mass centered at $0$, the temperature $T_{0}:=\int_{\R}x^{2}\mu_{0}(\d x)$ is positive  and one can define a sequence $(F_{0}^{n})_{n})$ such that 
\begin{equation} \label{eq:limF}
\lim_{n \to \infty}\int_{\R} \varphi(x)F_{0}^{n}(x)\d x=\int_{\R} \varphi(x)\mu_{0}(\d x) \qquad \forall\, \varphi \in L^{\infty}_{-\gmm}(\R) \cap \C(\R)\,.
\end{equation}
with $F_{0}^{n} \in \bigcap_{s \geq 0} L^{1}_{s}(\R)$ for any $n \geq 1$ (notice that, as in \cite{Lu}, $F_{0}^{n}$ is some slight modification of the Mehler transform of $\mu_{0}$). Then, according to Theorem \ref{theo:L1}, for any $n \geq 0$, there exists a family $(F_{t}^{n})_{t \geq 0} \subset L^{1}_{\gmm}(\R)$ such that $(\mu_{t}^{n})_{t \geq 0}$ is a weak measure solution to \eqref{eq:1} associated to $\mu_{0}^{n}$, where 
$$\mu_{t}^{n}(\d x)=F_{t}^{n}(x)\d x \qquad \forall\, t \geq 0\,.$$
Then, noticing  that
$$\|\mu_{t}^{n}\|_{\gamma} \leq \|\mu_{t}^{n}\|_{\gmm}  \leq \|\mu_{0}^{n}\|_{\gmm} \qquad \forall \,n \geq 1,$$
one easily checks that
\begin{equation*} \label{eq:Qmut^{n}}
\big|\la \Q(\mu_{t}^{n},\mu_{t}^{n})\,;\,\varphi\ra\big|\leq 2\|\varphi\|_{\infty}\|\mu_{0}^{n}\|_{\gmm}^{2}\,\qquad \forall\, t \geq 0,\,n \geq 1\,.
\end{equation*}
from which one deduces as in \cite{Lu} that there exists $C=C(\mu_{0}) > 0$ (depending only on $\|\mu_{0}\|_{\gmm}$) such that, for any $t_{2} \geq t_{1} \geq 0$, 
\begin{equation}\label{eq:conPhi}\sup_{n \geq 1}\bigg|\int_{\R}\varphi(x)\mu_{t_{1}}^{n}(\d x)-\int_{\R}\varphi(x)\mu_{t_{2}}^{n}(\d x)\bigg| \leq C(\mu_{0}) \,\|\varphi\|_{\infty}\,\big|t_{2}-t_{1}\big| \qquad \forall\, \varphi \in \C_{b}(\R)\,.\end{equation}
Moreover, on the basis of the \emph{a priori} estimates \eqref{eq:decay2} (see also Remark \ref{nb:gene-rho}), 
$$M_{k}(F_{t}^{n}) \leq C_{k}(\gamma,\|\mu_{0}^{n}\|_{0}) \|\mu_{0}^{n}\|_{\gmm}\,t^{-\frac{k-2}{\gamma}} \qquad \forall\,k > 2\,.$$
Since $\lim_{n \to \infty}\|\mu_{0}^{n}\|_{0}=1$ and $\lim_{n \to \infty}\|\mu_{0}^{n}\|_{\gmm}=\|\mu_{0}\|_{\gmm}$ according to \eqref{eq:limF}, we deduce that, for any $k > 2$, there exists some positive constant $\mathbf{C}_{k}$ depending only on $k$, $\gamma$, and $M_{\gmm}(\mu_{0})$ such that
$$\sup_{n \geq 1} M_{k}(F_{t}^{n}) \leq \mathbf{C}_{k}t^{-\frac{k-2}{\gamma}} \qquad \forall\, k>2\,.$$
From this, we conclude, as in \cite{Lu}  that there exists a subsequence (still denoted by) $(\mu_{t}^{n})_{t \geq 0}$ and a family $(\mu_{t})_{t \geq 0} \subset \M_{\gmm}^{+}(\R)$ such that
\beq\label{conv_mun}
\lim_{n \to \infty}\int_{\R}\varphi(x)\mu_{t}^{n}(\d x)=\int_{\R}\varphi(x)\mu_{t}(\d x) \qquad \forall\, \varphi \in \C_{c}(\R), \qquad t \geq 0\,,
\eeq 
\begin{equation*} \label{eq:mutlim}
\|\mu_{t}\|_{\gmm}\leq \|\mu_{0}\|_{\gmm}\,,\qquad M_{k}(\mu_{t}) \leq \mathbf{C}_{k}t^{-\frac{k-2}{\gamma}} \qquad \forall\, t > 0,\,k > 2\,,
\end{equation*}
and \eqref{eq:conPhi} still holds for the limit $\mu_{t}$ (which implies that, for  any $\varphi \in \C_{b}(\R)$, the mapping 
$t \in [0,\infty) \mapsto \int_{\R}\varphi(x)\mu_{t}(\d x)$ is continuous).
To prove that $(\mu_{t})_{t\geq 0}$ is a measure weak solution to \eqref{eq:1} associated to $\mu_{0}$ in the sense of Definition \ref{defi:weak}, one argues exactly as in \cite[Section 4]{Lu}. Finally, the fact that $\mu_{0}\in \bigcap_{k\geq 0}\P^{0}_{k}(\R)$ implies that $(\mu_{t})_{t\geq 0}\subset \bigcap_{k\geq 0}\P^{0}_{k}(\R)$ is proven as in  Theorem \ref{theo:L1}.\end{proof}

\begin{nb}
Arguing as in \cite{Lu}, it is not difficult to prove that any weak measure solution to \eqref{eq:1} is in fact a strong solution in the sense of \cite{Lu}.
\end{nb}
\begin{proof}[Proof of Proposition \ref{prop:stab}] 
Let $T > 0$ be fixed. For any $\phi \in \text{Lip}_{1}(\mathbb{R})$ and $t \in [0,T]$ define
\begin{equation*}
W_{\phi}(t):=\int_{\mathbb{R}}\phi(x)\mu_{t}(\text{d} x)-\int_{\mathbb{R}}\phi(x)\nu_{t}(\text{d} x)\,.
\end{equation*}
We will also use in the proof the notation
\begin{equation*}
W(t)=d_{\text{KR}}(\mu_{t},\nu_{t})=\sup_{\phi \in \text{Lip}_{1}(\mathbb{R})}W_{\phi}(t)\,,
\end{equation*}
and recall that
\begin{equation*}
W(t)=\int_{\mathbb{R}^{2}}|x-y|\pi_{t}(\text{d} x,\text{d} y) \; \text{ for some } \pi_{t} \in \Pi(\mu_{t},\nu_{t})\,.
\end{equation*}
Let now $\phi \in \text{Lip}_{1}(\mathbb{R})$ be fixed. One has
\begin{align}
\label{e1}
\begin{split}
\dfrac{\text{d} }{\text{d} t}W_{\phi}(t)&=\frac{1}{2}\int_{\mathbb{R}^{2}}|x-y|^{\gamma}\Delta \phi(x,y)\mu_{t}(\text{d} x)\mu_{t}(\text{d} y)\\
&\hspace{2cm}-\frac{1}{2}\int_{\mathbb{R}^{2}}|v-w|^{\gamma}\Delta\phi(v,w)\nu_{t}(\text{d} v)\nu_{t}(\text{d} w)\\
&\hspace{-1cm}= \int_{\mathbb{R}^{2}}\pi_{t}(\text{d} x,\text{d} v)\int_{\mathbb{R}^{2}}\Big[|x-y|^{\gamma}{\Delta}_0\phi(x,y)-|v-w|^{\gamma}{\Delta}_0\phi(v,w)\Big]\pi_{t}(\text{d} y,\text{d} w)
\end{split}
\end{align} 
where ${\Delta}_0\phi(x,y)=\phi(\a x + \b y)-\phi(x)$ for any $(x,y) \in \R^{2}.$ Now,
\begin{align}
\label{e2}
\begin{split}
|x-y|^{\gamma}{\Delta}_0\phi(x,y) & - |v-w|^{\gamma}{\Delta}_0\phi(v,w)=\\
\Big(|x-y|^{\gamma} &- |v-w|^{\gamma}\Big){\Delta}_0\phi(x,y) + |v-w|^{\gamma}\Big({\Delta}_0\phi(x,y) - {\Delta}_0\phi(v,w)\Big)\,,
\end{split}
\end{align}
and, recalling that the  Lipschitz constant of $\phi$ is at most one, the second term readily yields
\begin{align}
\label{e3}
\begin{split}
\Big|{\Delta}_0\phi(x,y)  - {\Delta}_0\phi(v,w)\Big| & = \Big|\phi(\a x + \b y) - \phi(\a v + \b w)-\phi(x)+\phi(v)\Big|\\
&\hspace{-.5cm}\leq (1+a)|x-v| + b|y-w|\,.
\end{split}
\end{align}
For the first term in \eqref{e2},we use the identity $A=\min\{A,B\}+(A-B)_{+}$, valid for any $A,B\geq0$, to obtain the estimate
\begin{align*}
\big|A^{\gamma}-B^{\gamma}|A&=\big|A^{\gamma}-B^{\gamma}|\big(\min\{A,B\}+(A-B)_{+}\big)\\
&\leq \big|A^{\gamma}-B^{\gamma}|\,\min\{A,B\}+\big|A^{\gamma}-B^{\gamma}|\,\big|A-B|\\
&\leq (1+\gamma)\max\{A,B\}^{\gamma}\big|A-B\big|\,.
\end{align*}
The last inequality follows noticing that $\big|A^{\gamma}-B^{\gamma}|\,\min\{A,B\}\leq \gamma\max\{A,B\}^{\gamma}\big|A-B\big|$\,.  Since $\big|\Delta_{0}\phi(x,y)\big|\leq b|x-y|$ we can choose $A=|x-y|$ and $B=|v-w|$ to conclude that
\begin{align}
\label{e4}
\begin{split}
\Big(|x-y|^{\gamma} - |v-w|^{\gamma}\Big)&{\Delta}_0\phi(x,y) \leq b(1+\gamma)\max\{|x-y|,|v-w|\}^{\gamma}\big| |x-y| -  |v-w|\big|\\
&\leq b(1+\gamma)\max\{|x-y|,|v-w|\}^{\gamma}\big( |x-v| +  |y-w|\big)\,.
\end{split}
\end{align}
Gathering the estimates \eqref{e2},\eqref{e3} and \eqref{e4} in \eqref{e1} and using symmetry of the expression, it follows that $\dfrac{\text{d} }{\text{d} t}W_{\phi}(t)\leq 3(1+\gamma)H(t)$ where we introduced
\begin{align}\label{e5}
H(t):=\int_{\mathbb{R}^{2}}\pi_{t}(\text{d} x,\text{d} v)\int_{\mathbb{R}^{2}}\big(|x|^{\gamma}+|y|^{\gamma}+|v|^{\gamma}+|w|^{\gamma}\big)|x-v|\pi_{t}(\text{d} y,\text{d} w)\,.
\end{align}
Expand $H(t)=H_{1}(t) + H_{2}(t)$ where
\begin{align*}
H_{1}(t):&= \int_{\mathbb{R}^{2}}\pi_{t}(\text{d} x,\text{d} v)\int_{\mathbb{R}^{2}}\big(|y|^{\gamma}+|w|^{\gamma}\big)|x-v|\pi_{t}(\text{d} y,\text{d} w)\,,\\
H_{2}(t):&= \int_{\mathbb{R}^{2}}\big(|x|^{\gamma}+|v|^{\gamma}\big)|x-v|\pi_{t}(\text{d} x,\text{d} v)\,.
\end{align*}
Notice that
\begin{align}
\label{e6}
\begin{split}
H_{1}(t)&=\int_{\mathbb{R}^{2}}\big(|y|^{\gamma}+|w|^{\gamma}\big)\pi_{t}(\text{d} y,\text{d} w)\,\int_{\mathbb{R}^{2}}|x-v|\pi_{t}(\text{d} x,\text{d} v)\\
&=W(t)\int_{\mathbb{R}}|x|^{\gamma}(\mu_{t}+\nu_{t})(\text{d} x) \leq C\,W(t)\,.
\end{split}
\end{align}
The last inequality follows because the weak measure solutions $\mu_{t}$ and $\nu_{t}$ have the  $\gamma$-moment uniformly bounded in $t\in[0,T]$, and additionally, $\pi_{t}\in\Pi(\mu_{t},\nu_{t})$ achieves the Kantorovich-Rubinstein distance. We estimate now $H_{2}(t)$ as in \cite[Corollary 2.3]{FM}. Namely, for any $t \in [0,T]$ and any $r > 0$, one has
\begin{equation*}
\begin{split}
\int_{\R^{2}}\bigg(|x|^{\gamma} + |v|^{\gamma} \bigg) \, |x-v|\,\pi_{t}&(\d x,\d v)\leq 
2r^{\gamma}\int_{\R^{2}} |x-v|\,\pi_{t}(\d x,\d v) \\
&\phantom{+++} +\int_{\min(|x|,|v|) \geq r}\bigg(|x|^{\gamma} + |v|^{\gamma} \bigg) \, |x-v|\,\pi_{t}(\d x,\d v)\\
&=2r^{\gamma}W(t)+ \int_{\min(|x|,|v|) \geq r}\bigg(|x|^{\gamma} + |v|^{\gamma} \bigg) \, |x-v|\,\pi_{t}(\d x,\d v)\,,
\end{split}
\end{equation*}
since $\pi_{t} \in \Pi(\mu_{t},\nu_{t})$ achieves the Kantorovich-Rubinstein distance. Setting now $R_{\varepsilon}$ such that
\begin{align*}
\Big(|x|^{\gamma} + |v|^{\gamma} \Big) \, & |x-v| \left(\exp\left(\frac{\varepsilon|x|^{\gamma}}{2}\right) + \exp\left(\frac{\varepsilon |v|^{\gamma}}{2}\right)\right) \\
&\leq R_{\varepsilon}\left(\exp(\varepsilon |x|^{\gamma})+ \exp(\varepsilon |v|^{\gamma})\right) \qquad \forall\, (x,v) \in \R^{2}\,,
\end{align*}
it follows that
$$H_{2}(t) \leq 2r^{\gamma} W(t) +  R_{\varepsilon}\,C_{T}(\varepsilon) \exp\left(-\frac{\varepsilon r^{\gamma}}{2}\right)\,.$$
 Choosing 
$$r^{\gamma}=\left|2\log W(t)/\varepsilon\right|$$
we obtain
\begin{equation}\label{e7}
H_{2}(t) \leq \dfrac{4}{\varepsilon} W(t)\,|\log W(t)| + R_{\varepsilon}\,C_{T}(\varepsilon)\,W(t)\,.
\end{equation}
Estimates \eqref{e6} and \eqref{e7} imply that
\begin{equation}
\label{e8}
\dfrac{\text{d} }{\text{d} t}W_{\phi}(t)\leq K_{\varepsilon}C_{T}(\varepsilon)\,W(t)\big(1 + \left|\log W(t)\right| \big)\,,
\end{equation}
with a constant $K_{\varepsilon} > 0$ depending only on $\gamma$ and $\varepsilon>0$.  Integrating \eqref{e8} and taking the supremum over $\phi \in \mathrm{Lip}_{1}(\R)$ we get the conclusion.
\end{proof}

\end{document}